\documentclass[reqno]{amsart}
\usepackage{amsthm, amsmath, amssymb}
\usepackage{color}
\usepackage{appendix}
\usepackage{hyperref}
\usepackage{enumitem}
\usepackage{graphicx}
\usepackage{caption}
\usepackage{subcaption}
\usepackage{cite}
\usepackage{mathrsfs}
\usepackage{soul}

\theoremstyle{theorem}
\newtheorem{theorem}{Theorem}
\newtheorem{lemma}{Lemma}
\newtheorem{corollary}{Corollary}

\theoremstyle{definition}
\newtheorem{definition}{Definition}

\theoremstyle{remark}
\newtheorem{remark}{Remark}

\newcommand{\ignore}[1]{}

\newcommand{\rev}[1]{#1}

\newcommand{\E}{\mathrm{\mathbf{E}}}
\newcommand{\1}{\mathrm{\mathbf{1}}}
\renewcommand{\P}{\mathrm{\mathbf{P}}}

\newcommand{\Q}{\mathrm{\mathbf{Q}}}
\newcommand{\sens}{\eta}

\newcommand{\Real}{\mathbb{R}}
\newcommand{\eps}{\varepsilon}
\newcommand{\abs}[1]{\left \lvert #1 \right \rvert}
\newcommand{\norm}[1]{\left \| #1 \right \|}

\DeclareMathOperator*{\adj}{adj}
\renewcommand{\t}{\mathrm{t}}

\title[Sharp entrywise perturbation bounds for Markov chains]{Sharp entrywise perturbation bounds for Markov chains}
\author{Erik Thiede, Brian Van Koten, Jonathan Weare}
\begin{document}
\begin{abstract}
For many Markov chains of practical interest,
the invariant distribution is extremely sensitive to perturbations of some entries of the transition matrix,
but insensitive to others;
we give an example of such a chain, motivated by a problem in computational statistical physics.
We have derived perturbation bounds on the relative error of the invariant distribution
that reveal these variations in sensitivity.

Our bounds are sharp, we do not impose any structural assumptions on the transition matrix or on the perturbation,
and computing the bounds has the same complexity
as computing the invariant distribution or computing other bounds in the literature.
Moreover, our bounds have a simple interpretation in terms of hitting times,
which can be used to draw intuitive but rigorous conclusions
about the sensitivity of a chain to various types of perturbations.
\end{abstract}
\maketitle

\section{Introduction}
The invariant distribution of a Markov chain is often extremely sensitive to perturbations of some entries
of the transition matrix, but insensitive to others.
However, most perturbation estimates bound the error in the invariant distribution
by a single condition number times a matrix norm of the perturbation.
That is, perturbation estimates usually take the form
\begin{equation}\label{eq: form single condition number}
\left \lvert \pi \left (\tilde{F}\right) - \pi(F) \right \rvert \leq \kappa(F) \left \lVert \tilde{F} - F \right \rVert,
\end{equation}
where $F$ and $\tilde{F}$ are the exact and perturbed
transition matrices, $\pi(F)$ and $\pi(\tilde{F})$ are the invariant distributions of these matrices,
$\lVert \cdot \rVert$ is a matrix norm,
$\lvert \cdot \rvert$ is either a vector norm or some measure of the relative error between the two distributions,
and $\kappa(F)$ is a condition number depending on $F$.
For example,
see~\cite{ChoMey:MeanFPT,GolMey:ComputingInvDist,IpsMey:UnifStab,FundMeyer:SensStationaryDist,HavivVDHeyden:PertBounds,
KirkNeuShade:PazIneqPertBounds,Meyer:Condition1980,Schweitzer:PertMarkov,Seneta:PertMeasuredByErgCfcts,
Seneta:SensitivityAnalysisRkOneUpdates} and the survey given in~\cite{ChoMey:Survey}.
No bound of this form can capture wide variations in the sensitivities of different entries of the transition matrix.

Alternatively, one might approximate the difference between $\pi$ and $\tilde{\pi}$
using the linearization of $\pi$ at $F$.
(The derivative of $\pi$ can be computed efficiently using the techniques of~\cite{GolMey:ComputingInvDist}.)
The linearization will reveal variations in sensitivities,
but only yields an approximation of the form
\begin{equation*}
  \pi\left (\tilde{F}\right) - \pi(F) = \pi'(F) \left [\tilde{F} - F \right] + o\left (\left \lVert \tilde{F}- F \right \rVert \right),
\end{equation*}
not an upper bound on the error.
That is, \rev{unless global bounds on $\pi'(F)$ can be derived},
linearization provides only a local, not a global estimate.

In this article, we give
upper bounds that yield detailed information about the sensitivity
of $\pi(F)$ to perturbations of individual entries of $F$.
Given an irreducible substochastic matrix $S$,
we show that for all stochastic matrices $F, \tilde{F}$
satisfying the entry-wise bound $F,\tilde{F} \geq S$,
\begin{align}
&\max_i \left \lvert \log \pi_i \left ( \tilde{F} \right ) - \log
\pi_i(F)\right \rvert \nonumber \\
&\qquad\qquad\leq \sum_{i \neq j}
 \left \lvert \log \left (\tilde{F}_{ij} + \Q_{ij}(S) - S_{ij} \right)
- \log(F_{ij} + \Q_{ij}(S)- S_{ij}) \right \rvert,
\label{eq: intro complicated version of bound}
\end{align}
where $\Q_{ij}(S)$ is defined in Section~\ref{sec: global bounds}.
As a corollary, we also have
\begin{equation}\label{eq: intro corollary bound}
\max_i \left \lvert \log \pi_i \left ( \tilde{F} \right ) - \log
\pi_i(F)\right \rvert
\leq \sum_{ i \neq j}
\Q_{ij}(S)^{-1} \left \lvert \tilde{F}_{ij} - F_{ij} \right \rvert
\end{equation}
for all stochastic $F,\tilde{F} \geq S$.

The difference in logarithms on the left hand sides
of~\eqref{eq: intro complicated version of bound}
and~\eqref{eq: intro corollary bound} measures
relative error.
Usually, when $\tilde{x} \in (0, \infty)$
is computed as an approximation to $x \in (0,\infty)$,
the error of $\tilde{x}$ relative to $x$ is defined to be either
\begin{equation}
  \left \lvert \frac{\tilde{x} - x}{x} \right \rvert \text{ or } \max \left \{\frac{\tilde x}{x}, \frac{x}{\tilde x} \right \}.
  \label{eq: intro definitions of relative error}
\end{equation}
Instead, we define the relative error to be
$\lvert \log \tilde{x} - \log x \rvert$.
Our definition is closely related to the other two:
it is the logarithm of the second definition
in~\eqref{eq: intro definitions of relative error};
and by Taylor expansion,
\begin{equation*}
\lvert \log \tilde{x} - \log x \rvert = \left \lvert \frac{\tilde{x} - x}{x} \right \rvert
+ O \left ( \left \lvert \frac{\tilde{x} - x}{x} \right \rvert^2 \right ),
\end{equation*}
so it is equivalent with the first in the limit of small error.
We chose our definition because it allows for simple
arguments based on logarithmic derivatives of $\pi(F)$.

We call the coefficient $\Q_{ij}(S)^{-1}$ in~\eqref{eq: intro corollary bound}
the \emph{sensitivity of the $ij^{\rm{th}}$ entry}.
$\Q_{ij}(S)$ has a simple probabilistic interpretation
in terms of hitting times,
which can sometimes be used to draw intuitive conclusions about
the sensitivities.
In Theorem~\ref{thm: sharpness},
we show that our coefficients $\Q_{ij}(S)^{-1}$
are within a factor of two of the smallest possible
so that a bound of form~\eqref{eq: intro corollary bound} holds.
Thus, our bound is sharp.
\rev{(We note that our definition of sharp differs slightly from other standard definitions;
see Remark~\ref{rem: comment on sharpness}.)}
In Theorem~\ref{thm: computing sensitivities}, we give
an algorithm by which the sensitivities may be computed in
$O(L^3)$ time for $L$ the number of states in the chain.
Therefore, computing the error bound has the same order
of complexity as computing the invariant measure
or computing most other perturbation bounds in the literature;
see Remark~\ref{rem: comparison of complexity}.

Since our result takes an unusual form,
we now give three examples to illustrate its use.
We discuss the examples only briefly here;
details follow in Sections~\ref{sec: global bounds} and~\ref{sec: hilly landscape}.
First, suppose that $\tilde{F}$ has been computed as an approximation
to an unknown stochastic matrix $F$
and that we have a bound on the error between $\tilde{F}$ and $F$,
for example $\lvert \tilde{F}_{ij} - F_{ij} \rvert \leq \alpha_{ij}$.
In this case, we define $S_{ij} := \max\{\tilde{F}_{ij} - \alpha_{ij},0\}$,
and we have the estimate
\begin{equation*}
\max_i \left \lvert \log \pi_i \left ( \tilde{F} \right ) - \log
\pi_i(F)\right \rvert
\leq \sum_{i \neq j} \Q_{ij} (S)^{-1}  \left \lvert \tilde{F}_{ij} - F_{ij} \right \rvert
\end{equation*}
for all $F$ so that $\lvert \tilde{F}_{ij} - F_{ij} \rvert \leq \alpha_{ij}$.
See Remark~\ref{rem: error estimate when true F is unknown} for a more detailed explanation.

Now suppose instead that $F \geq \alpha P$ where $0 < \alpha < 1$ and $P$
is the transition matrix of an especially simple Markov chain,
for example a symmetric random walk.
Then we choose $S := \alpha P$,
and we compute or approximate $\Q_{ij}(\alpha P)$ by
easy to understand probabilistic arguments.
This method can be used to draw intuitive
but rigorous conclusions about the sensitivity of a chain to various types of perturbations.
See Section~\ref{sec: bound below by random walk} for details.

Finally, suppose that the transition matrix $F$
has a large number of very small positive entries
and that we desire a sparse approximation $\tilde{F}$ to $F$
with approximately the same invariant distribution.
In this case, we take $S$ to be $F$ with all its small
positive entries set to zero.
If the sensitivity $\Q_{ij}(S)^{-1}$ is very large,
it is likely that the the value of $F_{ij}$ is important
and cannot be set to zero.
If $\Q_{ij}(S)^{-1}$ is small, then setting $\tilde{F}_{ij} =0$
and $\tilde{F}_{ii} = F_{ii} + F_{ij}$ will not have much effect on the invariant distribution.

We are aware of two other bounds on relative error in the literature.
By \cite[Theorem~4.1]{IpsMey:UnifStab},
\begin{equation}\label{eq: intro ipsen meyer bound}
 \left \lvert \frac{\pi_i(\tilde{F})- \pi_i(F)}{\pi_i(F)} \right \rvert
\leq \kappa_i (F) \left  \lVert \left (\tilde{F} - F \right )\left (I - e_i e_i^\t \right ) \right \rVert_\infty,
\end{equation}
where $\kappa_i(F) = \left \lVert (I - F_i)^{-1}  \right \rVert_\infty$
for $F_i$ the $i^{\rm{th}}$ principal submatrix of $F$.
This bound fails to identify the sensitive and insensitive entries
of the transition matrix since the error in the $i^{\rm{th}}$ component of the
invariant distribution is again controlled only by the single condition number $\kappa_i(F)$.
Moreover, computing
$\kappa_i(F)$ for all $i$ is of the same complexity
as computing all of our sensitivities $\Q_{ij}(S)^{-1}$;
see Remark~\ref{rem: comparison of complexity}.
Therefore, in many respects, our result provides more detailed information at the same cost
as~\cite[Theorem~4.1]{IpsMey:UnifStab}.
On the other hand, we observe that~\cite[Theorem~4.1]{IpsMey:UnifStab} holds for all perturbations:
one does not have to restrict the admissible perturbations by requiring $\tilde{F} \geq S$
as we do for our result.
However, this is not always an advantage,
since we anticipate that in many applications bounds on the error in $F$ are available and,
as we will see, the benefit from using this information is significant.

In~\cite[Theorem~1]{OC:RelErr}, another bound on the relative error is given.
Here, the relative error in the invariant distribution
is bounded by the relative error in the transition matrix.
Precisely, if $F,\tilde{F} \in \Real^{L \times L}$ are irreducible stochastic matrices
with $F_{ij} = 0$ if and only if $\tilde{F}_{ij} =0$,
then
\begin{equation}\label{eq: oc bound from intro}
\max_m \max \left \{\frac{\tilde{\pi}_m}{\pi_m}, \frac{\pi_m} {\tilde{\pi}_m}\right \}
\leq \left ( \max_{i \neq j}
\max \left \{\frac{\tilde{F}_{ij}}{F_{ij}}, \frac{F_{ij}} {\tilde{F}_{ij}}\right \} \right )^L
\end{equation}
This surprising result requires no condition number depending on $F$,
but it does require that $F$ and $\tilde{F}$ have the same sparsity pattern,
which greatly restricts the admissible perturbations.
Our result may be understood as a generalization of~\eqref{eq: oc bound from intro}
which allows perturbations changing the sparsity pattern.

Our result also bears some similarities with the analysis in~\cite[Section~4]{ChoMey:MeanFPT},
which is based on the results of~\cite{IpsMey:UnifStab}.
In~\cite{ChoMey:MeanFPT}, a state $m$ of a Markov chain is said to be \emph{centrally located}
if $\E_{i}[\tau_m]$ is small for all states $i$.
(Here, $\tau_m$ is the first passage time to state $m$;
see Section~\ref{sec: notation}.)
It is shown that if $|\E_{i}[\tau_m] - \E_j[\tau_m]|$ is small,
then $\pi_m(F)$ is insensitive to $F_{ij}$ in relative error.
Therefore, if $m$ is centrally located, $\pi_m(F)$ is not sensitive to any entry of the transition matrix.
Our $\Q_{ij}^{-1}$ can also be expressed in terms of first passage times,
and they provide a better measure of the sensitivity of $\pi_m(F)$ to $F_{ij}$ than
$|\E_{i}[\tau_m] - \E_j[\tau_m]|$;
see Section~\ref{sec: bounds by mean fpt}.

\rev{
Our bounds on derivatives of $\pi(F)$ in Theorem~\ref{thm: bound on log deriv}
and our estimates~\eqref{eq: intro complicated version of bound}
and~\eqref{eq: intro corollary bound}
share some features with structured condition numbers~\cite[Section~5.3]{KirklandNeumann:GroupInvMMatrices}.
The structured condition number of an irreducible, stochastic matrix $F$ is defined to be
\begin{equation*}
\begin{split}
  &\mathcal{C}^{(p,q)} (F) \\ &\qquad:= \lim_{\eps \rightarrow 0}
  \sup \left \{\eps^{-1} \left \lVert \pi(F+E) - \pi(F) \right \lVert_{p} :
  F + E \text{ stochastic and } \lVert E \rVert_q < \eps \right \}.
  \end{split}
\end{equation*}
Structured condition numbers yield approximate bounds valid for small perturbations.
These bounds are useful, since for small perturbations,
estimates of type~\eqref{eq: form single condition number} are often far too pessimistic.
We remark that our results~\eqref{eq: intro complicated version of bound}
and~\eqref{eq: intro corollary bound} give the user control over the size
of the perturbation through the choice of $S$.
(If $S$ is nearly stochastic, then only small perturbations are allowed.)
Therefore, like structured condition numbers, our results are good for small perturbations.
In addition, our results are true upper bounds, so they are more
robust than approximations derived from structured condition numbers.}

Our interest in perturbation bounds for Markov chains
arose from a problem in computational statistical physics;
we present a drastically simplified version below in Section~\ref{sec: hilly landscape}.
For this problem, the invariant distribution is extremely sensitive to some entries of
the transition matrix, but insensitive to others.
We use the problem to illustrate the differences between our result,~\cite[Theorem~1]{OC:RelErr},~\cite[Theorem~4.1]{IpsMey:UnifStab},
and the eight bounds on absolute error surveyed in~\cite{ChoMey:Survey}.
Each of the eight bounds has form~\eqref{eq: form single condition number},
and we demonstrate that the condition number $\kappa(F)$ in each bound
blows up exponentially with the inverse temperature parameter in our problem.
By contrast, many of the sensitivities $\Q_{ij}^{-1}$ from our result
are bounded as the inverse temperature increases.
Thus, our result gives a great deal more information about which
perturbations can lead to large changes in the invariant distribution.

\section{Notation}
\label{sec: notation}
We fix $L \in \mathbb{N}$,
and we let $X$ be a discrete time Markov chain
with state space $\Omega = \{1, 2, \dots, L\}$
and irreducible, row-stochastic transition matrix $F \in \Real^{L \times L}$.
Since $F$ is irreducible, $X$ has a unique invariant distribution $\pi \in \Real^{L}$ satisfying
\begin{equation*}
  \pi^\t F = \pi^\t \text{, } \sum_{i=1}^L \pi_i = 1 \text{, and } \pi_i >0 \text{ for all } i =1, \dots, L.
\end{equation*}

We let $e_i$ denote the $i^{\rm{th}}$ standard basis vector in $\Real^L$,
$e$ denote the vector of ones, and $I$ denote the identity matrix.
We treat all vectors, including $\pi$, as column vectors (that is, as $L \times 1$ matrices).
For $S \in \Real^{L\times L}$,
we let $S_j:e_j^\bot \rightarrow e_j^\bot$ be the operator defined by
\begin{equation*}
S_j x:= (I - e_j e_j^\t) S x.
\end{equation*}
Instead of defining $S_j$ as above, we could define $S_j$ to be $S$ with the $j^{\rm{th}}$
row and column set to zero.
We could also define $S_j$ to be the $j^{\rm{th}}$ principal submatrix.
We chose our definition to emphasize that we
treat $S_j$ as an operator on $e_j^\bot$.
If $S$ and $T$ are matrices of the same dimensions,
we say $S \geq T$ if and only if $S_{ij} \geq T_{ij}$ for all indices $i,j$.
\rev{For any $v \in \Real^L$ with $v > 0$, we define $\log(v) \in \Real^L$ by $(\log v)_i = \log(v_i)$.}

\rev{
For $k \in \{1,2,\dots, L\}$, we define $\1_k$ to be the indicator function of the set $\{k\}$,
and
\begin{equation*}
  \tau_k := \min \{s> 0 : X_s = k\}
\end{equation*}
to be the first return time to state $k$.
We also define
\begin{equation*}
\P_k[A] :=\P[A \vert X_0 = k] \text{ and } \E_k[Y] := \E[Y \vert X_0 = k]
\end{equation*}
to be the probability of the event $A$ conditioned on $X_0 = k$ and
the expectation of the random variable $Y$ conditioned on $X_0 = k$, respectively.
Finally, for $Y$ a random variable and $B$ an event, we let
\begin{equation*}
  \E[Y, B] := \E [Y \chi_B ] = \E [Y \vert B] \P [B],
\end{equation*}
where $\chi_B$ is the indicator function of the event $B$.
}

\section{Partial derivatives of the invariant distribution}
Given an irreducible, stochastic matrix $F \in \Real^{L \times L}$,
let $\pi(F) \in \Real^{L}$ be the invariant distribution of $F$;
that is, let $\pi(F)$ be the unique solution of
\begin{equation*}
  \pi(F)^\t F = \pi(F)^\t \mbox{ and } \pi(F)^\t e = 1.
\end{equation*}
We regard $\pi$ as a function defined on the set of irreducible stochastic matrices,
and in Lemma~\ref{lem: inv dist is differentiable}, we show that $\pi$ is differentiable
in a certain sense. We give a proof of the lemma in Appendix~\ref{app: proof inv dist is differentiable}.

\begin{lemma}\label{lem: inv dist is differentiable}
The function $\pi$ admits a continuously differentiable extension $\bar{\pi}$
to an open neighborhood $\mathcal{V}$
of the set of irreducible stochastic matrices in $\Real^{L \times L}$.
The extension may be chosen so that $\bar{\pi}(G) >0$ for all $G \in \mathcal{V}$
and so that if $Ge = e$, then
\begin{equation*}
  \bar{\pi} (G)^\t G = \bar{\pi} (G)^\t \mbox{ and } \bar{\pi}(G)^\t e = 1.
\end{equation*}
\end{lemma}

\begin{remark}
The set of stochastic matrices is not a vector space;
it is a compact, convex polytope lying in the affine space
$\{G \in \Real^{L \times L} : Ge = e\}\subset \Real^{L \times L}$.
As a consequence, we need the extension guaranteed by Lemma~\ref{lem: inv dist is differentiable}
to define the derivative of $\pi$ on the boundary of the polytope,
which is the set of all stochastic matrices with at least one zero entry.
We introduce $\bar{\pi}$ only to resolve this unpleasant technicality,
not to define $\pi(F)$ for matrices which are not stochastic.
In fact, all our results are independent of the particular choice of extension,
as long as it meets the conditions in the second sentence of the lemma.
\end{remark}

Our perturbation bounds are based on partial derivatives of $\pi$
with respect to entries of $F$.
As usual, the partial derivatives are defined in terms of a coordinate system,
and we choose the off-diagonal entries of $F$ as coordinates:
Any stochastic $F$ is determined by its off-diagonal entries through the formula
\begin{equation*}
F = I + \sum_{\substack{i, j \in \Omega \\ i \neq j}} F_{ij} \left (e_i e_j^\t - e_i e_i^\t \right ).
\end{equation*}
Accordingly, for $i,j \in \Omega$ with $i \neq j$, we define
\begin{align}
  \frac{\partial \pi_m}{\partial F_{ij}}(F) &=
\frac{\partial}{\partial F_{ij}} \bar{\pi}_m \left (I + \sum_{k \neq l} F_{kl} \left (e_k e_l^\t - e_k e_k^\t \right ) \right) \nonumber\\
&= \left . \frac{d}{d\eps} \right |_{\eps = 0} \bar{\pi}_m \left (F + \eps \left (e_i e_j^\t - e_i e_i^\t \right ) \right).
\label{eq: defn of partial derivatives}
\end{align}
\emph{These partial derivatives must be understood as derivatives
of the extension $\bar{\pi}$ guaranteed by Lemma~\ref{lem: inv dist is differentiable}.}
Otherwise, if $F_{ij}$ or $F_{ii}$ were zero,
the right hand side of~\eqref{eq: defn of partial derivatives} would be undefined.

\rev{
\begin{remark}
We chose to define partial derivatives by~\eqref{eq: defn of partial derivatives},
since that definition leads to the global bounds on the invariant distribution
presented in Section~\ref{sec: global bounds}.
Other choices are reasonable: for example, one might consider derivatives of the form
\begin{equation*}
\left . \frac{d}{d\eps} \right \rvert_{\eps = 0} \pi \left (F+\eps \left (e_i e_j^\t - e_i e_k^\t \right ) \right ),
\end{equation*}
where $k \neq i$ and $j \neq i$.
However, to the best of our knowledge, only definition~\eqref{eq: defn of partial derivatives}
leads easily to global bounds.
\end{remark}}

In Theorem~\ref{thm: derivative of invariant measure},
we derive a convenient formula for $\frac{\partial \pi_m}{\partial F_{ij}}$.
Comparable results relating derivatives and perturbations of $\pi$
to the matrix of mean first passage times
were given in~\cite{ChoMey:MeanFPT,Hunter:StationaryDistMeanFPT2005}.
\rev{A formula for the derivative of the invariant distribution in terms of the group inverse of $I-F$
was given in~\cite{Meyer:GroupInvMarkovChains1975};
a general formula for the derivative of the Perron vector of a nonnegative matrix was
given in~\cite{DeutschNeumann:DerivativesPerronVect1985}.}

\begin{theorem}\label{thm: derivative of invariant measure}
Let $F$ be an irreducible stochastic matrix,
and let $X$ be a Markov chain with transition matrix $F$.
Define $\pi(F)$ and $\frac{\partial \pi_m}{\partial F_{ij}}$ as above.
We have
\begin{align*}
\frac{\partial \pi_m}{\partial F_{ij}}(F)
&=
\pi_i \left (
\E_j \left [\sum_{s=0}^{\tau_i -1} \1_m(X_s) \right ] - \pi_m \E_j [\tau_i]
\right )
\end{align*}
for all $i,j \in \Omega$ with $i \neq j$.
\end{theorem}
\begin{proof}
Define $\bar{\pi}$ and $\mathcal{V}$ as in Lemma~\ref{lem: inv dist is differentiable},
let $v(G) := \bar{\pi}(G)/\bar{\pi}_i(G)$ for all $G \in \mathcal{V}$,
and let $G^\eps := F + \eps (e_i e_j^\t - e_i e_i^\t)$.
Define
\begin{equation*}
 \frac{\partial v}{\partial F_{ij}} (F)
:= \left (\frac{\partial v_i}{\partial F_{ij}} (F), \dots, \frac{\partial v_L}{\partial F_{ij}} (F) \right ) = \left .\frac{d}{d\eps} v(G^\eps) \right \rvert_{\eps = 0}.
\end{equation*}
Since $G^\eps e = e$,
Lemma~\ref{lem: inv dist is differentiable} implies
\begin{equation}\label{eq: properties of v}
v(G^\eps)^\t (I - G^\eps) = 0 \mbox{ and } v_i(G^\eps) = 1
\end{equation}
for all $\eps$ sufficiently close to zero.

We derive an equation for
$\frac{\partial v}{\partial F_{ij}} (F)$ from~\eqref{eq: properties of v};
differentiating\eqref{eq: properties of v} with respect to $\eps$ gives
\begin{equation}\label{eq: first equation satisfied by derivative}
 \frac{\partial v}{\partial F_{ij}} (F)^\t (I - F) = e_j^\t - e_i^\t  \mbox{ and }
\frac{\partial v_i}{\partial F_{ij}} (F) = 0.
\end{equation}
Recalling the definition of $F_i$ from Section~\ref{sec: notation},
~\eqref{eq: first equation satisfied by derivative} implies
\begin{equation*}
\frac{\partial v}{\partial F_{ij}}(F)^\t (I - F_i)
= e^\t_j.
\end{equation*}
Moreover, by~\cite[Chapter~6, Theorem~4.16]{BerPle},
\begin{equation}
\text{$F$ irreducible and stochastic or substochastic implies } (I-F_i)^{-1} = \sum_{s=0}^\infty F_i^s.
\label{eq: irred and stochastic implies convergent}
\end{equation}
Therefore, we have
\begin{equation}
\frac{\partial v}{\partial F_{ij}} = e^\t_j(I - F_i)^{-1}
= e^\t_j \sum_{k=0}^\infty F_i^k. \label{eq: power series for derivative of invariant measure}
\end{equation}

We now interpret~\eqref{eq: power series for derivative of invariant measure}
in terms of the Markov chain $X_t$ with transition matrix $F$.
We observe that for any $m \in \Omega \setminus \{i\}$,
\begin{align*}
e^\t_j  F_{i}^k e_m
&= \P_j[X_k = m, k < \tau_i],
\end{align*}
where $\tau_i := \min \{t > 0 : X_t = i\}$ is the first passage time to state $i$.
Therefore, for $m ,j \in \Omega\setminus \{i\}$,~\eqref{eq: power series for derivative of invariant measure} yields
\begin{equation}
\frac{\partial v_m}{\partial F_{ij}}
= \sum_{k = 0}^\infty \P_j[X_k = m, k < \tau_i]
=\E_j \left [\sum_{s=0}^{\tau_i -1} \1_m(X_s) \right ].
\label{eq: formula for derivatives of v}
\end{equation}
In fact, this formula also holds for $m = i$,
since we have
\begin{displaymath}
\frac{\partial v_i}{\partial F_{ij}}
= \E_j \left[ \sum_{s=0}^{\tau_i -1} \1_i(X_s) \right] = 0
\end{displaymath}
 for all $j \in \Omega \setminus \{i\}$.

Finally, we convert our formula for $\frac{\partial v_m}{\partial F_{ij}}$
to a formula for $\frac{\partial \pi_m}{\partial F_{ij}}$.
We have
\begin{equation*}
\pi_m = \frac{v_m}{\sum_{k=1}^L v_k},
\end{equation*}
 and so by~\eqref{eq: formula for derivatives of v},
\begin{align}
\frac{\partial \pi_m}{\partial F_{ij}} &=
\frac{\frac{\partial v_m}{\partial F_{ij}} \left (\sum_{k=1}^L v_k \right )
- v_m \sum_{k=1}^L \frac{\partial v_k}{\partial F_{ij}}}
{\left (\sum_{k=1}^L v_k \right )^2} \nonumber \\
&= \frac{\E_j \left [\sum_{s=0}^{\tau_i -1} \1_m(X_s) \right ] - \pi_m \E_j [\tau_i]}{\sum_k v_k}.
\label{eq: normalization of formula for derivative}
\end{align}
Now by~\cite[Theorem~1.7.5]{Nor:MarkovBook}, we have
\begin{equation}
  v_k(F) = \frac{\pi_k(F)}{\pi_i(F)} = \E_i \left [ \sum_{s=0}^{\tau_i-1} \1_k(X_s) \right ] \text{ and } \frac{1}{\pi_i(F)} = \E_i[\tau_i]. \label{eq: relation of inv dist to mean fpt}
\end{equation}
Therefore,~\eqref{eq: normalization of formula for derivative} implies
\begin{equation*}
\frac{\partial \pi_m}{\partial F_{ij}}
= \pi_i \left ( \E_j \left [\sum_{s=0}^{\tau_i -1} \1_m(X_s) \right ] - \pi_m \E_j [\tau_i] \right ).
\end{equation*}

\end{proof}

Our goal is to bound the relative errors of the entries
of the invariant measure $\pi(F)$,
where for $ \tilde{x},x \in (0, \infty)$,
we define the relative error between $\tilde{x}$ and ${x}$ to be
\begin{equation}\label{eq: defn of relative error}
  \lvert \log \tilde{x} - \log x \rvert.
\end{equation}
Our definition of relative error
is unusual, but it is closely related to the common definitions,
as we explain in the introduction.
In Theorem~\ref{thm: bound on log deriv},
we derive sharp bounds on the logarithmic partial derivatives of the invariant distribution.
We want bounds on logarithmic derivatives,
since we will ultimately prove bounds on the relative error in $\pi$,
with relative error defined by~\eqref{eq: defn of relative error}.

The following lemma will be used in the proof of Theorem~\ref{thm: bound on log deriv}.

\begin{lemma}\label{lem: decomposition of occupation time}
We have
\begin{align*}
 \P_i [\tau_j < \tau_i] \E_j \left [\sum_{s=0}^{\tau_i-1} \1_m(X_s) \right ]
&= \E_i \left [ \sum_{s = \tau_j}^{\tau_i - 1} \1_m(X_s) \right ],
\end{align*}
and
\begin{align*}
   \P_i [\tau_j < \tau_i] \E_j [\tau_i] &= \E_i [\tau_i - \tau_j, \tau_j < \tau_i].
\end{align*}
\end{lemma}

\begin{proof}
\ignore{
Define the process $Y_n := X_{\tau_j + n}$.
For $i \in \Omega$, let $\tau_i^X$ and $\tau_i^Y$
denote the first return times to $i$ for $X$ and $Y$, respectively.
We have
\begin{align*}
\E_i \left [ \sum_{s = \tau_j^X}^{\tau_i^X - 1} \1_m(X_s) \right ]
&= \E_i \left [ \sum_{s = \tau_j^X}^{\tau_i^X - 1} \1_m(X_s), \tau_j^X < \tau_i^X \right ] \\
&= \sum_{n=0}^\infty \P\left  [X_n = m, \tau_j^X \leq n < \tau_i^X \mid X_0 = i \right ] \\
&= \sum_{n=0}^\infty \P \left [Y_n = m, n < \tau_i^Y, \tau_j^X< \tau_i^X \mid X_0 = i \right ] \\
&= \sum_{n=0}^\infty \P \left [Y_n = m, n < \tau_i^Y \mid Y_0 = j \right ]
\P\left [\tau_j^X< \tau_i^X \mid X_0 = i \right ] \\
&= \P_i \left [\tau_j^X < \tau_i^X \right ] \E_j \left [\sum_{s=0}^{\tau_i-1} \1_m(X_s) \right ].
\end{align*}
The second-to-last step follows from the strong Markov property.

Summing over $m$, we also have
\begin{align*}
 \P_i [\tau_j < \tau_i] \E_j[\tau_i] =\E_i [\tau_i - \tau_j , \tau_j < \tau_i]
 \leq \E_i[\tau_i].
 \end{align*}
}
For $n \geq 0$, define  $Y_n := X_{\tau_j + n}$.
For $i \in \Omega$, let $\tau_i^X$ and $\tau_i^Y$
denote the first return times to $i$ for $X$ and $Y$, respectively.
By the strong Markov property,
(a) $Y$ is a Markov process with the same transition matrix as $X$,
(b) the distribution of $Y_0$ is $e_j^\t$,
and (c) conditional on $Y_0$, $Y_n$ is independent of $X_0, X_1, \dots, X_{\tau_j}$.
Therefore,
\begin{align*}
\E \left [ \sum_{s = \tau_j^X}^{\tau_i^X - 1} \1_m(X_s) \middle \vert X_0 =i\right ]
&=
\E
\left [ \sum_{s = \tau_j^X}^{\tau_i^X - 1} \1_m(X_s), \tau_j^X < \tau_i^X \middle \vert X_0 =i\right ] \\
&= \E
\left [ \sum_{s = 0}^{\tau_i^Y - 1} \1_m(Y_s), \tau_j^X < \tau_i^X \middle \vert X_0 =i \right ] \\
&= \E \left [ \sum_{s = 0}^{\tau_i^Y - 1} \1_m(Y_s) \middle \vert Y_0 = j \right ]
\P_i \left [\tau_j^X < \tau_i^X \right ].
\end{align*}
(The second equality above follows from the definition of $Y$;
the third equality follows from the strong Markov property,
since the event $\tau_j^X < \tau_i^X$ is determined by $X_0, X_1, \dots, X_{\tau_j}$.)
This proves the first formula in the statement of the lemma;
the second formula follows on summing the first over all $m$.
\end{proof}

Using Lemma~\ref{lem: decomposition of occupation time},
we now prove our bounds on the logarithmic derivatives.
\begin{theorem}
\label{thm: bound on log deriv}
We have
\begin{equation*}
\frac{1}{2} \frac{1}{\P_i [\tau_j < \tau_i]}
\leq \max_m \abs{\frac{\partial \log \pi_m}{\partial F_{ij}}}
\leq \frac{1}{\P_i [\tau_j < \tau_i]},
\end{equation*}
and
\begin{equation*}
\max_m  \frac{\partial \log \pi_m}{\partial F_{ij}} - \min_m  \frac{\partial \log \pi_m}{\partial F_{ij}} = \frac{1}{\P_i[\tau_j<\tau_i]}.
\end{equation*}
\end{theorem}
\begin{proof}
\ignore{
Recall
\begin{align*}
\frac{\pi_i}{\pi_m} \E_j \left [\sum_{s=0}^{\tau_i-1} \1_m(X_s) \right ]
&= \frac{1}{\P_i[\tau_j < \tau_i]}
 \frac{\pi_i}{\pi_m} \E_i \left [\sum_{s=\tau_j}^{\tau_i-1} \1_m(X_s) \right ] \\
 &= \frac{1}{\P_i[\tau_j < \tau_i]}
 - \frac{1}{\P_i[\tau_j < \tau_i]}  \frac{\pi_i}{\pi_m}
 \E_i \left [\sum_{s=0}^{\tau_j-1} \1_m(X_s) , \tau_j < \tau_i \right ] , \text{ and } \\
\pi_i \E_j[\tau_i] &=  \frac{1}{\P_i[\tau_j<\tau_i]} \pi_i
\E_i[\tau_i - \tau_j, \tau_j < \tau_i]. \\
\end{align*}
}

Using Theorem~\ref{thm: derivative of invariant measure},
Lemma~\ref{lem: decomposition of occupation time},
and~\eqref{eq: relation of inv dist to mean fpt},
we have
\begin{align}
\frac{\partial \log \pi_m}{\partial F_{ij}}
&= \frac{\pi_i}{\pi_m} \E_j \left [\sum_{s=0}^{\tau_i-1} \1_m(X_s) \right ]
-\pi_i \E_j[\tau_i]  \label{eq: log deriv formula one}\\
&= \frac{1}{\P_i[\tau_j<\tau_i]}
\frac{\pi_i}{\pi_m}
\E_i \left [\sum_{s=\tau_j}^{\tau_i-1} \1_m(X_s) \right ]
- \pi_i \E_j[\tau_i] \nonumber \\
&= \frac{1}{\P_i[\tau_j<\tau_i]}
\frac{\pi_i}{\pi_m}
\left (\E_i \left [\sum_{s=0}^{\tau_i-1} \1_m(X_s) \right ] -
\E_i \left [\sum_{s=0}^{\tau_j-1} \1_m(X_s), \tau_j < \tau_i \right ]
\right )  - \pi_i \E_j[\tau_i]
\nonumber \\
&= \frac{1}{\P_i[\tau_j<\tau_i]}
\left (
1 - \frac{\pi_i}{\pi_m}
 \E_i \left [\sum_{s=0}^{\tau_j-1} \1_m(X_s) , \tau_j < \tau_i \right ]
\right )
 - \pi_i \E_j[\tau_i]
. \label{eq: log deriv formula two}
\end{align}

We observe that the term
$\frac{\pi_i}{\pi_m} \E_j \left [\sum_{s=0}^{\tau_i-1} \1_m(X_s) \right ]$
in formula~\eqref{eq: log deriv formula one}
is nonnegative, hence that term attains a minimum value of zero when $m =i$.
Therefore, we have
\begin{align}
  \min_m  \frac{\partial \log \pi_m}{\partial F_{ij}}  &=
  -\pi_i \E_j [\tau_i],
  \label{eq: exact formula for min log deriv}
\end{align}
with the minimum attained when $m = i$,
since the other term in formula~\eqref{eq: log deriv formula one}
does not depend on $m$.
By a similar argument using~\eqref{eq: log deriv formula two},
\begin{align}
\max_m  \frac{\partial \log \pi_m}{\partial F_{ij}} &=
 \frac{1}{\P_i[\tau_j<\tau_i]} - \pi_i \E_j [\tau_i],
 \label{eq: exact formula for max log deriv}
\end{align}
and the maximum is attained when $m=j$.

Subtracting~\eqref{eq: exact formula for min log deriv}
from~\eqref{eq: exact formula for max log deriv} gives
\begin{equation*}
  \max_m  \frac{\partial \log \pi_m}{\partial F_{ij}} - \min_m  \frac{\partial \log \pi_m}{\partial F_{ij}} = \frac{1}{\P_i[\tau_j<\tau_i]},
\end{equation*}
hence
\begin{equation*}
\frac{1}{2} \frac{1}{\P_i [\tau_j < \tau_i]}
\leq \max_m \abs{\frac{\partial \log \pi_m}{\partial F_{ij}}}.
\end{equation*}
Finally, by Lemma~\ref{lem: decomposition of occupation time} and~\eqref{eq: relation of inv dist to mean fpt},
\begin{equation*}
0 \leq \pi_i \E_j [\tau_i]
= \frac{\pi_i\E_i [\tau_i - \tau_j, \tau_j < \tau_i]}{\P_i[\tau_j<\tau_i]}
\leq \frac{\pi_i \E_i [\tau_i]}{\P_i[\tau_j<\tau_i]}
= \frac{1}{\P_i[\tau_j<\tau_i]},
\end{equation*}
so
\begin{equation*}
\max_m \abs{\frac{\partial \log \pi_m}{\partial F_{ij}}}
\leq \frac{1}{\P_i [\tau_j < \tau_i]}.
\end{equation*}
\end{proof}

Corollary~\ref{cor: easy bound on log deriv} gives a simplified version of the bound in
Theorem~\ref{thm: bound on log deriv}.
This estimate can be used to derive~\cite[Theorem~1]{OC:RelErr},
which we have stated in equation~\eqref{eq: oc bound from intro} above.
We omit the proof.

\begin{corollary}\label{cor: easy bound on log deriv}
Whenever $F_{ij} \neq 0$,
\begin{align*}
\abs{\frac{\partial \log \pi_m}{\partial F_{ij}}} \leq \frac{1}{F_{ij}}.
\end{align*}
\end{corollary}

\begin{proof}
We have
\begin{equation*}
\P_i[\tau_j < \tau_i] \geq \P_i[X_1 = j] = F_{ij},
\end{equation*}
and so the result follows by Theorem~\ref{thm: bound on log deriv}.
\end{proof}

\section{Global perturbation bounds}\label{sec: global bounds}

In this section, we use our bounds on the derivatives
of the invariant distribution to prove global perturbation estimates.
Our estimates assume that both the exact transition matrix $F$
and the perturbed matrix $\tilde{F}$ are bounded below by some
irreducible substochastic matrix $S$.
As a consequence, coefficients $\Q_{ij}(S)$ depending on $S$ arise.

We define $\Q_{ij}(S)$
in terms of a Markov chain with transition matrix depending on $S$,
and our perturbation results are based on comparisons between this chain
and other chains with transition matrices $G \geq S$.
Therefore, to avoid confusion, we let
\begin{equation*}
  \P_i[A](G)
\end{equation*}
denote the probability that $X^G \in A$ conditioned on $X^G_0 = i$
for $X^G$ a chain with transition matrix $G$.
To give a specific example,
we intend $\P_i[\tau_i < \tau_j](G)$ to mean the
probability that $X^G$ hits $j$ before returning to $i$,
conditional on $X^G_0 = i$.

We now define $\Q_{ij}(S)$.
\begin{definition}\label{def: defn of Qij, probability interpretation}
For $S$ an irreducible and substochastic or stochastic matrix,
let $X^\omega$ be the Markov chain with state space
$\Omega \cup \{\omega\}$ and transition matrix
\begin{equation*}
S^\omega := \bordermatrix{~ &\Omega &\omega \cr \Omega &S &e - Se \cr \omega &0 &1}.
\end{equation*}
We think of $X^\omega$ as a chain with transition probability $S$,
but augmented by an absorbing state $\omega$ to adjust for the fact that $S$ is substochastic.
For all $i,j \in \Omega$ with $i \neq j$, we define
\begin{equation*}
\Q_{ij}(S) := \P_i[\tau_j< \min \{\tau_i, \tau_\omega\}](S^\omega).
\end{equation*}
\end{definition}

\begin{remark}
  We observe that for $F$ stochastic,
\begin{equation*}
\Q_{ij}(F) = \P_{i}[\tau_j < \min \{\tau_i, \tau_\omega\}](F^\omega)
= \P_{i}[\tau_j < \tau_i](F),
\end{equation*}
since the absorbing state $\omega$ does not communicate with the other states $\Omega$ when $F$ is stochastic.
\end{remark}
We now show that $\Q_{ij}(S)$ is monotone as a function of $S$.
This is the crucial step in deriving global perturbation bounds from
the bounds on derivatives in Theorem~\ref{thm: bound on log deriv}.

\begin{lemma}\label{lem: monotonicity}
Let $S$ be an irreducible substochastic matrix.
If $F$ is a stochastic or substochastic matrix with $F \geq S$,
then
\begin{equation*}
\P_i[\tau_j < \tau_i](F^\omega) = \Q_{ij}(F) \geq \Q_{ij}(S) > 0.
\end{equation*}
In addition, for any substochastic or stochastic matrix $S$,
\begin{equation*}
  \Q_{ij}(S) \geq S_{ij}.
\end{equation*}
\end{lemma}
\begin{proof}
Let $\mathscr{P}^M_{ij}$ be the set of all \rev{walks} of length $M$
in $\Omega$ which start at $i$, end at $j$,
and visit $i$ and $j$ only at the endpoints.
To be more precise, define
\begin{align*}
\mathscr{P}^M_{ij}&
:=  \\ &\{\gamma:\mathbb{Z}\cap[0,M] \rightarrow \Omega:
\gamma(0) = i, \gamma(M) = j, \gamma(k) \notin \{i,j\} \forall 0<k<M\}.
\end{align*}
We observe that since $F \geq S$,
\begin{align*}
\P_i[\tau_j < \tau_i](F^\omega)
&= \sum_{M=1}^\infty \sum_{\gamma \in \mathscr{P}^M_{ij}} \P_i [X_k = \gamma(k) \mbox{ for } k = 1, \dots, M](F^\omega) \\
&\geq \sum_{M=1}^\infty \sum_{\gamma \in \mathscr{P}^M_{ij}} \P_i [X_k = \gamma(k) \mbox{ for } k = 1, \dots, M](S^\omega)\\
&=\P_i[\tau_j< \min\{\tau_i, \tau_\omega\}](S^\omega) \\
&= \Q_{ij}(S).
\end{align*}
Since $S$ is irreducible, we also have that
\begin{equation*}
  \Q_{ij}(S) = \P_i[\tau_j<\min\{\tau_i, \tau_\omega\}](S^\omega) > 0.
\end{equation*}
Finally,
\begin{equation*}
  \Q_{ij}(S) = \P_i[\tau_j<\min\{\tau_i, \tau_\omega\}](S^\omega) \geq \P_i[X_1=j](S^\omega)
 = S_{ij},
\end{equation*}
which concludes the proof.
\end{proof}

Combining Theorem~\ref{thm: bound on log deriv} with Lemma~\ref{lem: monotonicity}
yields our global perturbation estimate.

\begin{theorem}\label{thm: global bound diff sparsity}
Let $F,\tilde{F}$ be stochastic matrices, let $S$ be substochastic and irreducible,
and assume that $F, \tilde{F} \geq S$.
We have
\begin{equation*}
  \lvert \log \pi_m(\tilde{F}) - \log \pi_m(F) \rvert
\leq \sum_{\substack{i,j \in \Omega \\ i \neq j}}
 \left \lvert \log \left (\tilde{F}_{ij} - S_{ij} + \Q_{ij}(S) \right)
- \log(F_{ij}- S_{ij} + \Q_{ij}(S)) \right \rvert.
\end{equation*}
\end{theorem}

\rev{
\begin{remark}
We allow $S$ to be stochastic in Definition~\ref{def: defn of Qij, probability interpretation}
and also in the hypotheses of Theorem~\ref{thm: global bound diff sparsity}.
However, we observe that if $S$ is stochastic,
then the conclusion of the theorem is trivial since $S$ is the unique stochastic $F$
with $F \geq S$.
\end{remark}}

\begin{proof}
Let $G \in \{t\tilde{F} + (1-t) F: t \in [0,1] \}$.
Since $\Q_{ij}(S) = \P_i[\tau_j < \min\{\tau_i, \tau_\omega\}](S^\omega)$, we have
\begin{align*}
\P_i[\tau_j < \tau_i](G) &= G_{ij}+\P_i[1<\tau_j<\tau_i](G) \\
&\geq G_{ij} + \P_i[1<\tau_j<\min\{\tau_i, \tau_\omega\}](S^\omega) \\
&= G_{ij} + \P_i[\tau_j<\min\{\tau_i, \tau_\omega\}](S^\omega)- S_{ij} \\
&= G_{ij} + \Q_{ij}(S) - S_{ij}.
\end{align*}
Therefore,
\begin{align*}
   \lvert \log \pi_m(\tilde{F}) - \log \pi_m(F) \rvert
   &=
   \left \lvert  \sum_{i \neq j} \int_0^1 \frac{\partial \log \pi_m}{\partial F_{ij}} \left (t\tilde{F} + (1-t) F \right )
   \left (\tilde{F}_{ij} - F_{ij} \right) \, dt\right \rvert \\
   &\leq \sum_{i \neq j} \int_0^1 \frac{\left \lvert \tilde{F}_{ij} - F_{ij} \right \rvert}
   {\P_i[\tau_j < \tau_i]\left (t\tilde{F} + (1-t) F \right)} \, dt \\
   &\leq \sum_{i \neq j}
   \int_0^1 \frac{\left \lvert \tilde{F}_{ij} - F_{ij} \right \rvert}{t \tilde{F}_{ij} + (1-t) F_{ij} + \Q_{ij}(S) - S_{ij} } \, dt \\
   &=
   \sum_{i \neq j} \left \lvert \log \left (\tilde{F}_{ij}  - S_{ij} + \Q_{ij}(S) \right) - \log(F_{ij}- S_{ij}  + \Q_{ij}(S) ) \right \rvert.
\end{align*}
The first equality holds since by Lemma~\ref{lem: inv dist is differentiable},
the invariant distribution $\pi$ is Fr\'echet differentiable on an open neighborhood
of the set of irreducible, stochastic matrices.
Directional derivatives can then be computed using partial derivatives as above.
The denominator in the third line is positive since
$F,\tilde{F} \geq S$ implies
\begin{equation*}
t\tilde{F}_{ij} + (1-t) F_{ij} - S_{ij} \geq 0,
\end{equation*}
and by Lemma~\ref{lem: monotonicity}, $\Q_{ij}(S) > 0$.
\end{proof}

Theorem~\ref{thm: global bound diff sparsity}
takes a somewhat complicated form,
so in Corollary~\ref{cor: simpler version of global bound},
we present a simplified version.
The proof of Corollary~\ref{cor: simpler version of global bound}
shows that the bound in Theorem~\ref{thm: global bound diff sparsity}
is always smaller than the bound in Corollary~\ref{cor: simpler version of global bound}.

\begin{corollary}\label{cor: simpler version of global bound}
Let $F,\tilde{F}$ be stochastic matrices, let $S$ be substochastic and irreducible,
and assume that $F, \tilde{F} \geq S$.
We have
\begin{equation*}
  \lvert \log \pi_m(\tilde{F}) - \log \pi_m(F) \rvert
\leq \sum_{\substack{i,j \in \Omega \\ i \neq j}}
 \frac{\left \lvert\tilde{F}_{ij} - F_{ij}\right \rvert}{\Q_{ij}(S)} .
\end{equation*}
\end{corollary}
\begin{proof}
We have
\begin{align*}
 &\left  \lvert \log \left (\tilde{F}_{ij} + \Q_{ij}(S) - S_{ij} \right)
- \log(F_{ij} + \Q_{ij}(S)- S_{ij}) \right \rvert \\
&\qquad\qquad\leq \max_{x \in [\Q_{ij}(S), \infty]} \frac{d \log}{dx}(x)
\left \lvert (\tilde{F}_{ij} - S_{ij}) - (F_{ij} - S_{ij}) \right \rvert\\
&\qquad\qquad= \frac{\left \lvert\tilde{F}_{ij} - F_{ij}\right \rvert}{\Q_{ij}(S)},
\end{align*}
and so the result follows directly from Theorem~\ref{thm: global bound diff sparsity}.
\end{proof}

We show in Theorem~\ref{thm: sharpness} that
both Theorem~\ref{thm: global bound diff sparsity} and
Corollary~\ref{cor: simpler version of global bound} are sharp.
That is, we show that $\rho_{ij}(S) = \Q_{ij}(S)$
is within a factor of two of the largest value of
$\rho_{ij}(S)$ so that a bound of the form
\begin{equation*}
\lvert \log \pi_m(\tilde{F}) - \log \pi_m(F) \rvert
\leq \sum_{\substack{i,j \in \Omega \\ i \neq j}}
\left \lvert
\log(\tilde{F}_{ij} - S_{ij} + \rho_{ij}(S)) - \log( F_{ij} - S_{ij}+ \rho_{ij}(S))
\right \rvert,
\end{equation*}
and we show that $\eta_{ij}(S) = \Q_{ij}(S)^{-1}$
is within a factor of two of the smallest value of $\eta_{ij}(S)$
such that a bound of the form
\begin{equation*}
  \lvert \log \pi_m(\tilde{F}) - \log \pi_m(F) \rvert
\leq \sum_{\substack{i,j \in \Omega \\ i \neq j}}
\eta_{ij}(S) \left \lvert\tilde{F}_{ij} - F_{ij}\right \rvert
\end{equation*}
holds.

\rev{
\begin{remark}\label{rem: comment on sharpness}
We note that some authors call a bound sharp if
it is possible for equality to hold.
For example, a bound of form~\eqref{eq: form single condition number}
may be called sharp if for every stochastic $F$,
there exists a stochastic $\tilde{F}$ so that equality holds.
We prefer to call a bound sharp if it is the best bound of a given form,
possibly up to a small constant factor.
\end{remark}}

\begin{theorem}\label{thm: sharpness}
Let $S$ be an \rev{irreducible}, substochastic matrix,
and let $i,j \in \Omega$ with $i \neq j$.
For every $\eps >0$, there exist stochastic matrices $\tilde{F},F$
with $\tilde{F},F \geq S$ so that
\begin{align*}
\max_{m \in \Omega}
&\left \lvert \log \pi_m \left (\tilde{F} \right ) - \log \pi_m(F) \right \rvert \\
&\qquad\geq
\frac{1}{2} (\Q_{ij}(S)+ \eps)^{-1} \left \lvert \tilde{F}_{ij} - F_{ij} \right \rvert
\\
&\qquad\geq
\left \lvert \log(\tilde{F}_{ij} - S_{ij} + 2(\Q_{ij}(S) + \eps) )
- \log(F_{ij} - S_{ij} + 2(\Q_{ij}(S) + \eps) )\right \rvert
\end{align*}
and $\left \lvert \tilde{F}_{kl} - F_{kl} \right \rvert = 0$
for all $(k,l)$ except $(i,j)$ and $(i,i)$.
\end{theorem}
\begin{proof}
Define
\begin{equation*}
  F_{kl} :=
\begin{cases}
  S_{kl} &\text{ if } l \neq i, \text{ and } \\
  1 - \sum_{m \neq i} S_{km} &\text{ if } l = i.
\end{cases}
\end{equation*}
$F$ is stochastic, $F \geq S$, $F_j - F_j e_i e_i^\t = S_j - S_j e_i e_i^\t$, and $F_{j^\bot, j} = S_{j^\bot, j}$.
Therefore,
\begin{equation*}
\Q_{ij}(S) = \Q_{ij}(F) = \P_i[\tau_j < \tau_i](F).
\end{equation*}

We now distinguish two cases: $\sum_{m \in \Omega} S_{im} = 1$ and $\sum_{m \in \Omega} S_{im} < 1$.
In the first case, if $F \geq S$ and $F$ is stochastic, then $F_{im} = S_{im}$ for all $m \in \Omega$.
Therefore, $\left \lvert \tilde{F}_{ij} - F_{ij} \right \rvert =0$ for all stochastic $F,\tilde{F} \geq S$,
 and so the conclusion of the theorem holds.
In the second case, we observe that $\sum_{m \in \Omega} S_{im} < 1$ implies $F_{ii} > S_{ii} \geq 0$.
Therefore, for any sufficiently small $\eta >0$,
\begin{equation*}
F^\eta := F + \eta(e_i e_j^\t - e_ie_i^\t)
\end{equation*}
is stochastic, and $F^\eta \geq S$.
By Theorem~\ref{thm: bound on log deriv}, we have
\begin{equation*}
\max_{m \in \Omega} \abs{\frac{\partial \log \pi_m}{\partial F_{ij}}(F)}
\geq \frac{1}{2} \frac{1}{\P_i [\tau_j < \tau_i]}
=\frac{1}{2}\frac{1}{\Q_{ij}(S)}.
\end{equation*}
It follows that for every $\eps >0$ there exists an $\eta >0$ with
\begin{align*}
&\max_{m \in \Omega} \left \lvert \log \pi_m \left (F^\eta \right ) - \log \pi_m(F) \right \rvert \\
&\quad\geq
\frac{1}{2} (\Q_{ij}(S) + \eps)^{-1} \left \lvert F^\eta_{ij} - F_{ij} \right \rvert \\
&\quad\geq
\left \lvert \log(F^\eta_{ij} - S_{ij}+2(\Q_{ij}(S)+\eps))
- \log(F_{ij}-S_{ij}+2(\Q_{ij}(S) + \eps))\right \rvert.
\end{align*}
(The second inequality follows by an argument
similar to the proof of Corollary~\ref{cor: simpler version of global bound}.)
\end{proof}

Theorem~\ref{thm: global bound diff sparsity} and Corollary~\ref{cor: simpler version of global bound}
take unusual forms,
and at first glance the condition $F,\tilde{F} \geq S$ may seem inconvenient.
In Remarks~\ref{rem: monotonicity and choice of S}
and~\ref{rem: error estimate when true F is unknown},
we explain how to apply these estimates.

\rev{
\begin{remark}\label{rem: monotonicity and choice of S}
For a given application, the best upper bounds are obtained by choosing the largest possible $S$.
This is a consequence of Lemma~\ref{lem: monotonicity}.
We apply this principle in Remark~\ref{rem: error estimate when true F is unknown}.
\end{remark}}

\begin{remark}\label{rem: error estimate when true F is unknown}
Suppose that $\tilde{F}$ has been computed as an approximation to an unknown stochastic matrix $F$
and that we have some bound on the error between $\tilde{F}$ and $F$.
\rev{For example, suppose that for some matrix $\alpha \geq 0$,
\begin{equation*}
\left \lvert \tilde{F}_{ij} - F_{ij} \right \rvert \leq \alpha_{ij} \mbox{ for all } i,  j \in \Omega.
\end{equation*}
In this case, we define
\begin{equation*}
S_{ij} := \max  \{\tilde{F}_{ij} - \alpha_{ij}, 0\} \mbox{ for all } i,  j \in \Omega.
\end{equation*}
We observe that this choice of $S$ is the largest possible so that $F \geq S$
for all $F$ with $\lvert F_{ij} - \tilde{F}_{ij} \rvert \leq \alpha_{ij}$.
Therefore, by Lemma~\ref{lem: monotonicity},
the coefficients $\Q_{ij}(S)$ are as large as possible,
giving the best possible upper bounds.
}

If $S$ is irreducible, we have
\begin{align*}
&\max_{k \in \Omega} \left \lvert \log \pi_k \left (\tilde{F} \right )
- \log \pi_k(F) \right \rvert\\
&\quad\leq \sum_{\substack{i,j \in \Omega \\ i \neq j}}
 \left \lvert \log \left (\tilde{F}_{ij}  - S_{ij}+ \Q_{ij}(S) \right)
- \log(F_{ij} - S_{ij} + \Q_{ij}(S)) \right \rvert\\
&\quad\leq \sum_{\substack{i,j \in \Omega \\ i \neq j}}
 \frac{\left \lvert\tilde{F}_{ij} - F_{ij}\right \rvert}{\Q_{ij}(S)}.
\end{align*}
In general, if $S$ is reducible,
then no statement can be made about the error of the invariant distribution.
In fact, if $S$ is reducible, then there is a reducible, stochastic $F$ with
$\left \lvert \tilde{F}_{ij} - F_{ij} \right \rvert \leq \alpha_{ij}$,
and the invariant distribution of $F$ is not even unique.
\end{remark}

\section{An efficient algorithm for computing sensitivities}
\label{sec: algorithm}
In Theorem~\ref{thm: computing sensitivities} below,
we show that the coefficients $\Q_{ij}(S)$
can be computed by inverting an $L\times L$ matrix
and performing additional operations of cost $O(L^2)$.
Therefore,
the cost of estimating the error in $\pi$ using
either Theorem~\ref{thm: global bound diff sparsity}
or Corollary~\ref{cor: simpler version of global bound}
is comparable to the cost of computing $\pi$.
Moreover, the cost of computing our bounds is the same as
the cost of computing most other bounds in the literature,
for example, those based on the group inverse;
see Remark~\ref{rem: comparison of complexity}.

Our first step is to characterize $\Q_{ij}(S)$ as the solution of a linear equation.
We advise the reader that we will make extensive use of
the notation introduced in Section~\ref{sec: notation}.
In addition, we define
\begin{equation}\label{def: s j bot j}
  S_{j^\bot, j} := (I - e_j e_j^\t)S e_j \text{ and } S_{j,j^\bot} := e_j^\t S (I - e_j e_j^\t)
\end{equation}
to be the $j^{\rm{th}}$ column and row of $S$ with the $j^{\rm{th}}$ entry set to zero,
respectively.

\begin{lemma}
\label{lem: computing transition prob}
Let $S$ be irreducible and substochastic or stochastic.
For $i,j \in \Omega$ with $i \neq j$,
let $q^{ij}(S) \in e_j^\bot$ be the vector defined by
\begin{equation*}
 q_{k}^{ij}(S)  := \P_k[\tau_j< \min \{\tau_i, \tau_\omega\}](S^\omega)
\end{equation*}
for all $k \in \Omega \setminus \{j\}$.
Define $S_j$ as in Section~\ref{sec: notation},
and $S_{j^\bot, j}$ by~\eqref{def: s j bot j}.
The operator $I - S_j + S_j e_i e_i^\t$ is invertible on $e_j^\bot$,
and $q^{ij}(S)$ is the unique solution of the equation
\begin{equation}\label{eq: matrix eqn for pijk}
\left ( I - S_j + S_j e_i e_i^\t \right ) q^{ij}(S)  =  S_{j^\bot, j} .
\end{equation}
\end{lemma}

\begin{proof}
Let $i,j \in  \Omega$ with $i \neq j$.
We have
\begin{align*}
\P_k [ \tau_j < \min \{\tau_i, \tau_\omega\}]
 = \sum_{l \in \Omega \cup \{\omega\}} \P_k [ \tau_j < \min \{\tau_i, \tau_\omega\} , X_1 = l ],
\end{align*}
and for $k \neq j$,
\begin{align*}
\P_k [ \tau_j < \min \{\tau_i, \tau_\omega\}  , X_1 = l ] =
\begin{cases}
0 &\mbox{ if } l \in \{i,\omega\}, \\
S_{kj}
&\mbox{ if } l =j, \mbox{ and }\\
\P_l[\tau_j <  \min \{\tau_i, \tau_\omega\}]S_{kl} &\mbox{ if } l \notin \{i,j, \omega\}.
\end{cases}
\end{align*}
Therefore,
\begin{equation}\label{eq: nonmatrix eq for pijk}
\P_k [\tau_j  <  \min \{\tau_i, \tau_\omega\}] =
\sum_{l \notin \{i,j,\omega\}} S_{kl} \P_l [\tau_j  <  \min \{\tau_i, \tau_\omega\}] + S_{kj}
\mbox{ for } k \neq j.
\end{equation}
We observe that equation~\eqref{eq: nonmatrix eq for pijk}
above can be expressed as
\begin{equation}\label{eq: linear eqn for trans prob}
q^{ij} = (S_j - S_j e_i e_i^\t) q^{ij} + S_{j^\bot, j}.
\end{equation}

We now claim that if $S$ is irreducible, then $I - S_j + S_j e_i e_i^\t$ is invertible,
which shows that $q^{ij}$ is the unique solution of~\eqref{eq: linear eqn for trans prob}.
By~\eqref{eq: irred and stochastic implies convergent},
$(I-S_j)^{-1} = \sum_{m=0}^\infty S_j^m$.
We now observe that
\begin{equation*}0 \leq  S_j - S_j e_i e_i^\t \leq S_j, \end{equation*}
so $\sum_{m=0}^\infty (S_j - S_j e_i e_i^\t)^m$ converges,
hence $I - S_j + S_j e_i e_i^\t$ is invertible.
\end{proof}

The proof of Theorem~\ref{thm: computing sensitivities} uses the following lemma.
\begin{lemma}\label{lem: algorithm lemma}
For $S$ substochastic and irreducible,
\begin{equation*}
\Q_{ij}(S)=\frac{e_i^\t (I - S_j)^{-1} S_{j^\bot, j}}{e_i^\t (I-S_j)^{-1} e_i}.
\end{equation*}
\end{lemma}
\begin{proof}
  We recall that
\begin{equation*}
 (I - S_j + S_j e_i e_i^\t)q^{ij}(S) =  S_{j^\bot, j }.
\end{equation*}
Multiplying both sides by $(I-S_j)^{-1}$ then yields
\begin{equation*}
(I+(I - S_j)^{-1}S_j e_i e_i^\t) q^{ij} (S) = (I - S_j)^{-1}S_{j^\bot, j}.
\end{equation*}
Now $S_j \geq 0$ is a substochastic matrix,
so by~\eqref{eq: irred and stochastic implies convergent},
${(I-S_j)^{-1} = \sum_{m=0}^\infty S_j^m \geq 0}$.
Therefore, $1+e_i^\t (I - S_j)^{-1}S_j e_i \geq 1$,
and by the Sherman-Morrison formula,
\begin{align*}
\left (I+(I - S_j)^{-1}S_j e_i e_i^\t \right)^{-1}
&= I -
\frac{(I - S_j)^{-1}S_j e_i e_i^\t}{1+e_i^\t (I - S_j)^{-1}S_j e_i} \\
&= I - \frac{(I - S_j)^{-1}S_j e_i e_i^\t}{e_i^\t (I-S_j)^{-1} e_i}.
\end{align*}
Thus,
\begin{align}
\Q_{ij}(S) = q^{ij}_i(S) &= e_i^\t (I+(I - S_j)^{-1}S_j e_i e_i^\t)^{-1} (I - S_j)^{-1}S_{j^\bot, j} \nonumber \\
&= \frac{e_i^\t (I - S_j)^{-1} S_{j^\bot, j}}{e_i^\t (I-S_j)^{-1} e_i}.
\end{align}
\end{proof}

\begin{theorem}
\label{thm: computing sensitivities}
Let $S \in \Real^{L \times L}$ be irreducible and substochastic.
The set of all coefficients $\Q_{ij}(S)$
can be computed by inverting an $L\times L$ matrix
and then performing additional operations of cost $O(L^2)$.
\end{theorem}

\begin{proof}
We begin with some definitions and notation.
Define
\begin{equation*}
  A(j) := I - S + e_j e_j^\t S =
\bordermatrix{~ &e_j^\bot & e_j \cr
                 e_j^\bot & I-S_j & -S_{j^\bot, j} \cr
             e_j & 0 & 1}.
\end{equation*}
\rev{(The right hand side above denotes the block decomposition of $A$ with respect
to the decomposition $\Real^L \simeq e_j^\bot \oplus \Real e_j$.
Thus, for example, $I-S_j$ is to be interpreted as an operator on $e_j^\bot$;
cf.\ the definition of $S_j$ in Section~\ref{sec: notation}.)}
We observe that $A(j)$ is invertible,
since by~\eqref{eq: irred and stochastic implies convergent}, $I-S_j$ is invertible, so
\begin{equation}\label{eq: algorithm formula for AjInv}
  A(j)^{-1} = \bordermatrix{~ &e_j^\bot &e_j \cr
                        e_j^\bot & (I-S_j)^{-1} & (I-S_j)^{-1} S_{j^\bot, j} \cr
                        e_j & 0 & 1}.
\end{equation}

In the first step of the algorithm, we compute $A(1)^{-1}$,
which costs $O(L^3)$ operations.
Second, we compute $SA(1)^{-1}$.
Given $A(1)^{-1}$, this can be done in $O(L^2)$ operations
using the formula
\begin{align*}
  SA(1)^{-1}
\ignore{ &= \bordermatrix{~ &e_1^\bot &e_1 \cr
                        e_1^\bot & S_1(I-S_1)^{-1} &  S_1 (I-S_1)^{-1} S_{1^\bot, 1} + S_{1^\bot, 1}  \cr
                        e_1 & S_{1,1^\bot} (I-S_1)^{-1} & S_{11} +S_{1,1^\bot} (I-S_1)^{-1} S_{1^\bot, 1} } \\
&= \bordermatrix{~ &e_1^\bot &e_1 \cr
                        e_1^\bot & (I-S_1)^{-1} - I & (I-S_1)^{-1}S_{1^\bot,1}  \cr
                        e_1 & S_{1,1^\bot} (I-S_1)^{-1} &S_{11} +S_{1,1^\bot} (I-S_1)^{-1} S_{1^\bot, 1}  } \\}
&=\bordermatrix{~ &e_1^\bot &e_1 \cr
                        e_1^\bot & A(1)_1^{-1} - I & A(1)^{-1}_{1^\bot,1}  \cr
                        e_1 & S_{1,1^\bot} A(1)_1^{-1} &S_{11} +S_{1,1^\bot} A(1)^{-1}_{1^\bot,1}}.
\end{align*}
(The formula is easily proved by direct calculation using~\eqref{eq: algorithm formula for AjInv}.)
Third, we compute $\Q_{i1}(S)$ for all $i \neq 1$.
By Lemma~\ref{lem: algorithm lemma} and~\ref{eq: algorithm formula for AjInv}, we have
\begin{equation}\label{eq: Q in terms of A1Inv}
  \Q_{i1}(S) = \frac{e_i^\t (I - S_1)^{-1} S_{1^\bot, 1}}{e_i^\t (I-S_1)^{-1} e_i}
= \frac{A(1)^{-1}_{i1}}{A(1)^{-1}_{ii}}.
\end{equation}
Therefore, once $A(1)^{-1}$ has been computed,
it costs $O(L)$ operations to compute $\Q_{i1}(S)$ for all $i \neq 1$.

We compute the remaining sensitivities $\Q_{ij}(S)$ for $j \neq 1$ by a formula analogous
to~\eqref{eq: Q in terms of A1Inv},
but with $A(j)^{-1}$ in place of $A(1)^{-1}$.
To do so efficiently,
we use the Sherman-Morrison-Woodbury identity to derive a formula expressing $A(j)^{-1}$ in terms of $A(1)^{-1}$:
\begin{align}
A(j)^{-1} &= \left (I - S + e_1 e_1^\t S + (e_j e_j^\t S - e_1 e_1^\t S) \right )^{-1} \nonumber\\
&= A(1)^{-1}
- A(1)^{-1} \begin{pmatrix} e_j &-e_1 \end{pmatrix}
C(j)^{-1}
\begin{pmatrix} e_j^\t \\ e_1^\t \end{pmatrix} SA(1)^{-1},  \label{eq: algorithm formula two}
\end{align}
where
\begin{equation*}
C(j) :=  I
+ \begin{pmatrix} e_j^\t \\ e_1^\t \end{pmatrix} S A(1)^{-1}
\begin{pmatrix} e_j &-e_1 \end{pmatrix}.
\end{equation*}
In the fourth step of the algorithm, we loop over all $j \neq 1$.
For each $j$, we first compute $C(j)^{-1}$,
which requires a total of $O(L)$ operations.
We then compute $A(j)^{-1}e_j$ at a cost of $O(L)$ operations using a formula derived from~\eqref{eq: algorithm formula two}:
\begin{align*}
  A(j)^{-1}e_j &= A(1)^{-1} e_j
- A(1)^{-1} \begin{pmatrix} e_j &-e_1 \end{pmatrix}
C(j)^{-1}
\begin{pmatrix} e_j^\t \\ e_1^\t \end{pmatrix} SA(1)^{-1}e_j \\
&= A(1)^{-1} e_j
- \begin{pmatrix}A(1)^{-1}  e_j &-A(1)^{-1} e_1 \end{pmatrix}
C(j)^{-1}
\begin{pmatrix} SA(1)^{-1}_{jj}  \\ SA(1)^{-1}_{1j} \end{pmatrix}.
\end{align*}
Next, we must compute $A(j)^{-1}_{ii}$ for all $i \neq j$.
By~\eqref{eq: algorithm formula two},
\begin{align*}
  A(j)^{-1}_{ii} = A(1)^{-1}_{ii} -  \begin{pmatrix} A(1)^{-1}_{ij} &-A(1)^{-1}_{i1} \end{pmatrix}
C(j)^{-1}
\begin{pmatrix}SA(1)^{-1}_{ji} \\ SA(1)^{-1}_{1i} \end{pmatrix},
\end{align*}
so the cost of computing $A(j)^{-1}_{ii}$ for each $i \neq j$ is $O(L)$.
Finally, we compute $\Q_{ij}(S)$ for all $i \neq j$.
By Lemma~\ref{lem: algorithm lemma} and~\ref{eq: algorithm formula for AjInv}, we have
\begin{equation}\label{eq: Q in terms of A1Inv}
  \Q_{ij}(S) = \frac{e_i^\t (I - S_j)^{-1} S_{j^\bot, j}}{e_i^\t (I-S_j)^{-1} e_i}
= \frac{A(1)^{-1}_{ij}}{A(1)^{-1}_{ii}},
\end{equation}
so this last step costs $O(L)$.

The total cost of the algorithm described above is
a single $L \times L$ matrix inversion plus $O(L^2)$.
\end{proof}

\begin{remark}\label{rem: comparison of complexity}
Most perturbation bounds in the literature have the same computational complexity as our bound.
For example, some bounds are based on the group inverse of $I-F$~\cite{Meyer:Condition1980,KirkNeuShade:PazIneqPertBounds}.
The cost of computing the group inverse
is $O(L^3)$~\cite{GolMey:ComputingInvDist},
so our bound has the same complexity as~\cite{Meyer:Condition1980,KirkNeuShade:PazIneqPertBounds}.
Computing the bound on relative error in~\cite{IpsMey:UnifStab} requires finding
$\left \lVert (I-F_j)^{-1} \right \rVert_\infty$ for all $j$; see~\eqref{eq: intro ipsen meyer bound}.
This could be done in $O(L^3)$ operations by methods similar to Theorem~\ref{thm: computing sensitivities},
so we conjecture that our bound and the bound of~\cite{IpsMey:UnifStab} have the same complexity.
On the other hand, the bound on relative error in~\cite{OC:RelErr}
(see~\eqref{eq: oc bound from intro}) requires almost no calculation at all.
\end{remark}

\rev{
\begin{remark}
We give the algorithm above to show that
the cost of computing our bounds is comparable to the cost of other bounds, in principle.
We do not claim that the algorithm is always reliable,
since we have not performed a complete stability analysis.
Nonetheless, in many cases, the computation of $\Q_{ij}(S)$ is stable
even when the computation of $\pi(F)$ is unstable.
For example, suppose that $S = \alpha F$ for $F$ a stochastic matrix
and $\alpha \in (0,1)$.
(This would be a good choice of $S$ if all entries of $F$ were known with relative error $\alpha^{-1}$;
cf.\ Remark~\ref{rem: error estimate when true F is unknown}.)
Let $\lVert M \rVert_\infty$ denote the operator norm of the matrix $M \in \Real^{L \times L}$
with respect to the $\ell^\infty$-norm
\begin{equation*}
  \lVert v \rVert_\infty := \max_{i = 1, \dots, L} \lvert v_i \rvert \text{ for } v \in \Real^L.
\end{equation*}
It is a standard result that
\begin{equation*}
  \lVert M \rVert_\infty = \max_{i = 1, \dots, L} \sum_{j=1}^L \lvert M_{ij} \rvert.
\end{equation*}
Therefore, since $F$ is stochastic and $0 \leq S_j \leq S = \alpha F$,
we have $\lVert S_j \rVert_\infty \leq \alpha$,
and so
\begin{equation*}
  \left \lVert (I - S_j)^{-1} \right \rVert_\infty
  = \left \lVert \sum_{n=0}^\infty S_j^n \right \rVert_\infty
  \leq \sum_{n=0}^\infty \lVert S_j^n  \rVert_\infty
  \leq \sum_{n=0}^\infty \alpha^n = \frac{1}{1-\alpha}.
\end{equation*}
Moreover,
\begin{equation*}
  \lVert I -S_j \rVert_\infty \leq 2,
\end{equation*}
so the condition number for the inversion of $I - S_j$ satisfies
\begin{equation*}
  \kappa_\infty ( I- S_j) :=
  \lVert I- S_j \rVert_\infty
  \lVert (I- S_j)^{-1} \rVert_\infty
  \leq \frac{2}{1-\alpha}.
\end{equation*}
We conclude that if $\alpha$ is not too close to one,
then the algorithm is stable.
For example, if the entries of $F$ are known with $2\%$ error,
then we choose $S = 0.98 F$, and we have $\kappa_\infty(F) \leq 100$.
Note that this estimate holds for any $F$,
no matter how unstable the computation of $\pi$ may be.
\end{remark}}

\section{The hilly landscape example}
\label{sec: hilly landscape}
In this section, we discuss an example in which
the invariant distribution is very sensitive
to some entries of the transition matrix,
but insensitive to others.
The example arose from a problem in computational statistical physics.
We will use the example to compare our results with previous work,
especially~\cite{ChoMey:MeanFPT,IpsMey:UnifStab,OC:RelErr}
and the bounds on absolute error summarized in~\cite{ChoMey:Survey}.

\subsection{Transition matrix and physical interpretation}
Our hilly landscape example is a simple analogue of the dynamics
of a single particle in contact with a heat bath.
Define $V:\Real \rightarrow \Real$ by
\begin{equation*}
  V(x) = \frac{1}{4 \pi} \cos(4 \pi x)
\end{equation*}
Take $L \in \mathbb{N}$,
and let $\Omega := \{1,2, \dots, L\}$ with periodic boundary conditions;
that is, take $\Omega := \mathbb{Z} / L \mathbb{Z}$.
Given $V$, we define a probability distribution on $\Omega$ by
\begin{equation*}
  \pi(i) := \frac{\exp (- LV(i/L))}{\sum_{k = 1}^L \exp (- LV(k/L))}.
\end{equation*}
The measure $\pi$ is in detailed balance with the Markov chain $X$ having transition matrix
$F \in \Real^{\Omega \times \Omega}$ defined by
\begin{alignat*}{3}
F_{ii} &:= \frac{1}{2} \left (\frac{\pi(i)}{\pi(i-1) + \pi(i)} + \frac{\pi(i)}{\pi(i+1) + \pi(i)} \right )
&&\mbox{ for all } i\in \Omega, \\
F_{i, i+1} &:= \frac{1}{2}\frac{\pi(i+1)}{\pi(i+1) + \pi(i)} &&\mbox{ for all } i\in \Omega, \\
F_{i, i-1} &:= \frac{1}{2}\frac{\pi(i-1)}{\pi(i-1) + \pi(i)} &&\mbox{ for all } i\in \Omega, \mbox{ and }\\
F_{ij} &:= 0 &&\mbox{ otherwise.}
\end{alignat*}
(In the definition of $F$, $F_{L,L+1}$ means $F_{L,1}$ and $F_{1,0}$ means $F_{1,L}$,
since we take $\Omega$ with periodic boundary conditions.)

We interpret $X$ as the position of a particle which moves through the interval $(0,1]$
with periodic boundary conditions.
If $V(i) > V(j)$, we say that $j$ is \emph{downhill} from $i$.
When the inequality is reversed we say that $j$ is \emph{uphill} from $i.$
Under the dynamics prescribed by $F$,
the particle is more likely to move downhill than uphill.
In fact, as $L$ tends to infinity,
$\pi$ becomes more and more concentrated near the minima of $V$.
For large $L$,
the particle spends most of the time near minima of $V$,
and transitions of the particle  between minima occur rarely.

\subsection{Sensitivities for the hilly landscape transition matrix}
\label{sec: calculations and figs}
In Figure~\ref{fig: sensitivities}, we plot $-\log \Q_{ij}(\alpha F)$
versus $i$ and $j$ for $F$ the hilly landscape transition matrix with $L=40$
and $\alpha \in \{ 0.7,0.8,0.9,0.95,0.98,1\}$.
The purpose of this section is to give an intuitive explanation of the main features observed in the figure.
Recall that the potential $V$ is shaped roughly like a ``W'' with peaks at $0$, $\frac{1}{2}$, and $1$
and valleys at $\frac{1}{4}$ and $\frac{3}{4}$.
When $L=40$, the peaks correspond to the indices $0$, $20$, and $40$ in $\Omega$,
and the valleys correspond to $10$ and $30$.
(To be precise, $0$ and $40$ are identical since we take periodic boundary conditions.)

Now consider the case $\alpha =1$.
We observe that $-\log \Q_{20,j}(F)$ is small for all $j$,
so $\Q_{20,j}(F)^{-1}$ is small,
and $\pi(F)$ is insensitive to perturbations which change the
transition probabilities from the peak to other points.
This is as expected, since the probability $\P_{20}[\tau_j<\tau_{20}](F) = \Q_{20,j}(F)$
of hitting a point $j$ in the valley before returning to the peak should be fairly large.
On the other hand, $-\log \Q_{30,10}(F)$ is enormous,
so $\pi(F)$ is sensitive to the transition probability from the valley to the peak.
This is also as expected, since the probability $\P_{30}[\tau_{10} < \tau_{30}](F)$
of climbing from the valley to the peak without falling back into the valley should be small.
To explain the small values of $-\log \Q_{ij}(F)$ observed near the diagonal,
we observe that for all $i \in \Omega$ and all $L$,
\begin{equation*}
  F_{i,i+1} = \frac{1}{2} \frac{1}{1 + \exp\left [ L \left ( V \left (\frac{i+1}{L} \right ) - V \left (\frac{i}{L} \right ) \right ) \right ]}
\geq \frac{1}{2} \frac{1}{1+\exp({\rm Lip}(V))} = \frac{1}{2} \frac{1}{1+\exp(1)},
\end{equation*}
where ${\rm Lip}(V)=1$ is the Lipschitz constant of the potential $V$.
\footnote{\rev{A function $f : \Real\rightarrow \Real$ is Lipschitz if for some
$K \geq 0$, $\lvert f(x) - f(y) \rvert \leq K \lvert x - y\rvert$ for all $x,y \in \Real$.
The Lipschitz constant, ${\rm Lip}(f)$, of a function $f$
is the smallest $K$ for which the bound in the last sentence holds.}}
Therefore, by Corollary~\ref{cor: easy bound on log deriv},
\begin{equation*}
\Q_{i,i+1}(F)^{-1} \leq \frac{1}{F_{i,i+1}} \leq 2(1 + \exp(1)).
\end{equation*}
The same estimate holds for $\Q_{i,i-1}(F)$.

\ignore{
In Figure~\ref{fig: kappa for alpha F}, we plot $\sens_{ij}(0.95 F)=\frac{1}{\P_i[\tau_j < \tau_i]}$
versus $i$ and $j$.
These sensitivities would be relevant, for example, if all entries of $F$ were known with $5\%$
accuracy; see Remark~\ref{rem: error estimate when true F is unknown}.}

The coefficients $\Q_{ij}(\alpha F)$ for $\alpha < 1$ share many features with $\Q_{ij}(F)$.
These coefficients would be relevant if $F$ were known with relative error $1-\alpha$;
see Remark~\ref{rem: error estimate when true F is unknown}.
The main difference between $\Q_{ij}(\alpha F)$ and $\Q_{ij}(F)$ is that $\Q_{ij}(\alpha F)$
is small whenever the minimum number of time steps
required to transition between $i$ and $j$ is large.
The reader is directed to  Section~\ref{sec: bound below by random walk} for discussion of a related phenomenon.
Observe that this effect grows more dominant as $\alpha$ decreases.
We also note that $\Q_{ij}(\alpha F)^{-1}$ is again small near the diagonal.
In fact, by Lemma~\ref{lem: monotonicity}, we have
\begin{equation*}
  \Q_{i,i+1}(\alpha F)^{-1} \leq \frac{1}{\alpha F_{i,i+1}} \leq \frac{2(1+\exp(1))}{\alpha}.
\end{equation*}
The same estimate holds for $\Q_{i,i-1}(\alpha F)$.

\begin{figure}
\caption{Sensitivities for the hilly landscape transition matrix.}
\label{fig: sensitivities}
\begin{subfigure}{0.45\textwidth}
\includegraphics{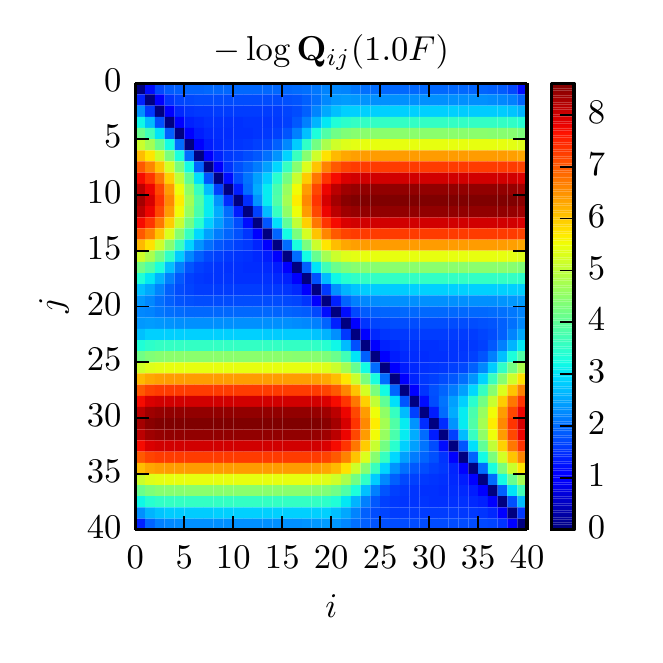}
\end{subfigure}
\qquad
\begin{subfigure}{0.45\textwidth}
\includegraphics{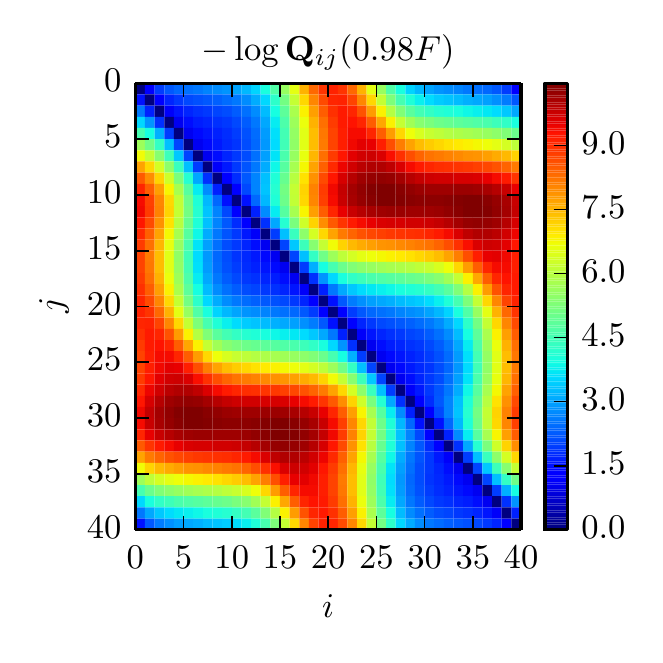}
\end{subfigure}
\\
\begin{subfigure}{0.45\textwidth}
\includegraphics{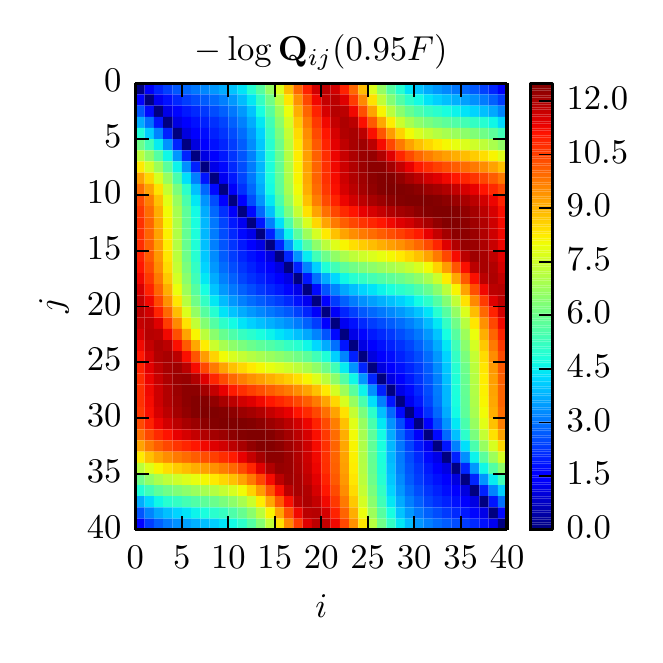}
\end{subfigure}
\qquad
\begin{subfigure}{0.45\textwidth}
\includegraphics{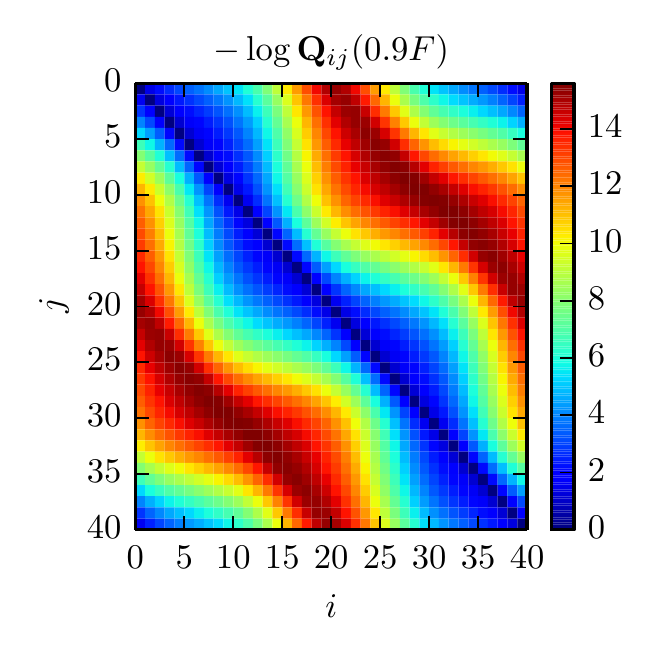}
\end{subfigure}
\\
\begin{subfigure}{0.45\textwidth}
\includegraphics{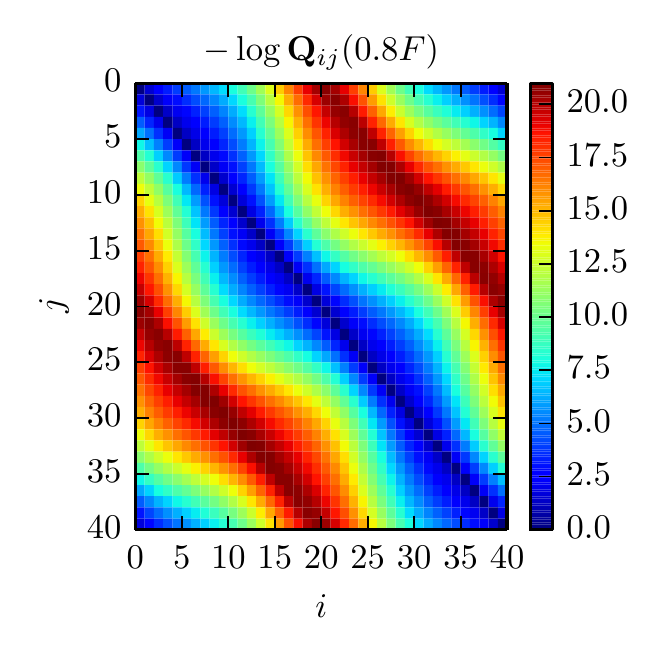}
\end{subfigure}
\qquad
\begin{subfigure}{0.45\textwidth}
\includegraphics{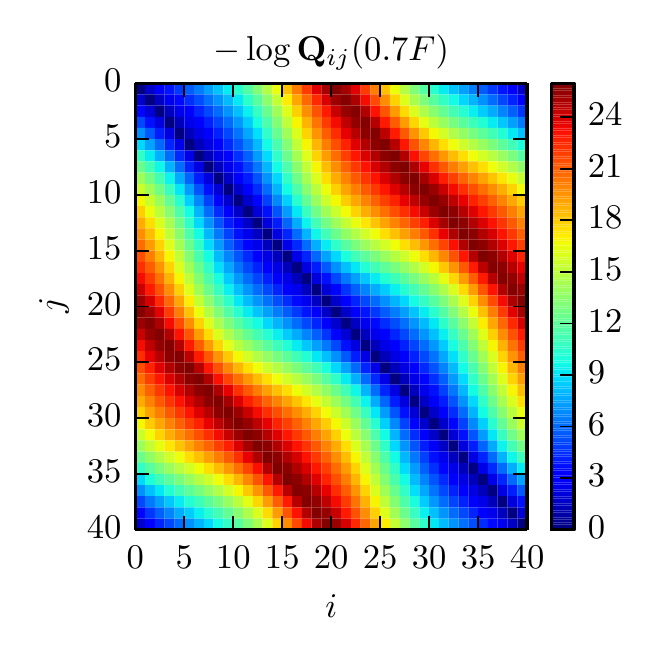}
\end{subfigure}
\end{figure}

\subsection{Mean first passage times and related relative error bounds}
\label{sec: bounds by mean fpt}

Section~4 of~\cite{ChoMey:MeanFPT} also suggests bounds on relative error in terms of certain first passage times.
A comparison is therefore in order.
We record a simplified version of these results below.

\begin{theorem}{\cite[Corollaries 4.1,4.2]{ChoMey:MeanFPT}}\label{thm: cho meyer result}
Let $F$ and $\tilde{F}$ be irreducible stochastic matrices.
We have
\begin{equation*}
\frac{\pi_m(\tilde{F}) - \pi_m(F)}{\pi_m(F)}
= \sum_{\substack{i,j \in \Omega \\ i \neq j}}
\pi_i(\tilde{F})((1-\delta_{jm})\E_j[\tau_m]-(1 - \delta_{im}) \E_i[\tau_m])
\left (\tilde{F}_{ij} - F_{ij} \right ),
\end{equation*}
where the expectations are taken for the chain with transition matrix $F$ and $\delta$ is the Kronecker delta function.
Therefore,
\begin{equation*}
  \left \lvert \frac{\pi_m(\tilde{F}) - \pi_m(F)}{\pi_m(F)}\right \rvert \leq \sum_{\substack{i,j \in \Omega \\ i \neq j}}
\lvert (1-\delta_{jm})\E_j[\tau_m]-(1 - \delta_{im}) \E_i[\tau_m]\rvert
\left \lvert \tilde{F}_{ij} - F_{ij} \right \rvert.
\end{equation*}
\end{theorem}

Taking the maximum over all $m$ in the second sentence of Theorem~\ref{thm: cho meyer result}
yields a pertubation result having a form similar to
Theorem~\ref{thm: global bound diff sparsity},
but with
\begin{equation*}
  \beta_{ij}(F) := \max_{m\in\Omega} \lvert (1-\delta_{jm})\E_j[\tau_m]-(1 - \delta_{im}) \E_i[\tau_m]\rvert
\end{equation*}
in place of $\Q_{ij}(S)^{-1}$.
In the next paragraph, we show for the hilly landscape transition matrix that for some values of $i$ and $j$,
$\beta_{ij}$ grows exponentially with $L$ while $\Q_{ij}(S)^{-1}$ remains bounded.
Thus, the results of~\cite{ChoMey:MeanFPT} dramatically overestimate the error due to some perturbations.

To derive the estimate in the second sentence of Theorem~\ref{thm: cho meyer result}
from the exact formula in the first sentence, one discards a factor of $\pi_i(\tilde{F})$.
Therefore, roughly speaking, the estimate is poor when $\pi_i(\tilde{F})$ is small.
To give a specific example,
let $F$ be the hilly landscape transition matrix,
assume that $L$ is even, and let $i=L/2$ and $j = L/2+1$.
Observe that we have chosen $i$ at a peak of the potential $V$, so
\begin{equation*}
\pi_i(F) \gtrsim  \exp(-L)
\end{equation*}
(Here, we use the symbols $\lesssim$ and $\gtrsim$ to denote bounds up to multiplicative constants,
so in the last line above, we mean that there is some $C>0$ so that the left hand side is bounded above by $C\exp(-L)$ for all $L$.)
Let $F^\eps = F + \eps (e_i e_j^\t - e_i e_i^\t)$.
We have
\begin{equation*}
  \frac{\partial \log \pi_m}{\partial F_{ij}}(F) = \left . \frac{d}{d\eps} \right \vert_{\eps = 0} \pi_m(F^\eps)
= \pi_i(F)((1-\delta_{jm})\E_j[\tau_m]-(1 - \delta_{im}) \E_i[\tau_m])
\end{equation*}
by Theorem~\ref{thm: cho meyer result}.
Therefore, by Theorem~\ref{thm: bound on log deriv},
\begin{align*}
  \max_{m \in \Omega} \left \lvert \frac{\partial \log \pi_m}{\partial F_{ij}}(F) \right \rvert
&= \pi_i(F) \max_{m \in \Omega} \lvert (1-\delta_{jm})\E_j[\tau_m]-(1 - \delta_{im}) \E_i[\tau_m] \rvert \\
&\geq \frac{1}{2} \frac{1}{\P_i[\tau_j < \tau_i]},
\end{align*}
and so
\begin{equation*}
\beta_{ij}(F) \geq  \frac{1}{2} \frac{1}{\pi_i(F)\P_i[\tau_j < \tau_i]}
\gtrsim \exp(L).
\end{equation*}
Now suppose that the substochastic matrix $S$ appearing in our bound is chosen for each $L$
so that $S_{ij}$ for $i=L/2$ and $j = L/2+1$ is bounded above zero uniformly
as $L \rightarrow \infty$.
For example, one might choose $S$ to be a multiple of $F$ as in
Section~\ref{sec: calculations and figs} or
a multiple of a simple random walk transition matrix as in
Section~\ref{sec: bound below by random walk}.
Then by Lemma~\ref{lem: monotonicity}, we have
\begin{equation*}
\Q_{ij}(S)^{-1} \leq \frac{1}{S_{ij}},
\end{equation*}
so $\Q_{ij}(S)^{-1}$ is bounded.
Thus, $\beta_{ij}(F)$ is a poor estimate of the sensitivity of the $ij^{\rm{th}}$ entry for this problem.

\subsection{The spectral gap and related absolute error bounds}
\label{sec: comparison with cho meyer survey}
The survey article~\cite{ChoMey:Survey} lists eight condition numbers $\kappa_i(F)$
for $i = 1, 2, \dots, 8$ for which bounds of the form
\begin{equation*}
  \left \lVert \pi(F) - \pi(\tilde{F}) \right \rVert_p \leq \kappa_i(F) \norm{F - \tilde{F}}_{p'}
\end{equation*}
hold.
(The H\"{o}lder exponents $p$ and $p'$ vary with the choice of condition number.)
\rev{Some of these condition numbers are based on ergodicity coefficients~\cite{Seneta:PertMeasuredByErgCfcts,Seneta:SensitivityAnalysisRkOneUpdates},
some on mean first passage times~\cite{ChoMey:MeanFPT},
and some on generalized inverses of the characteristic matrix
$I-F$~\cite{Meyer:Condition1980,Schweitzer:PertMarkov,FundMeyer:SensStationaryDist,IpsMey:UnifStab,HavivVDHeyden:PertBounds}.}
We prove for the hilly landscape transition matrix $F$ that $\kappa_i(F)$
increases exponentially with $L$ for all $i$.
By contrast, we have already seen that many of the coefficients
$\Q_{ij}(\alpha F)^{-1}$ are bounded as $L$ tends to infinity.

Our proof that the condition numbers increase exponentially is based on
an analysis of the spectral gap of $F$.
Let $\sigma(F)$ denote the spectrum of $F$.
The spectral gap $\gamma$ is defined to be
\begin{equation}\label{eq: defn of spectral gap}
\gamma := 1 - \max\{\lvert \lambda \rvert : \lambda \in \sigma(F) \setminus \{1\} \}.
\end{equation}
We use the bottleneck inequality~\cite[Theorem~7.3]{LevPerWil:Mixing} to show that the spectral gap
of the hilly landscape transition matrix decreases exponentially with $L$.
For convenience, assume that $L$ is even, and let
\begin{equation*}
E := \left \{1, \dots, \frac{L}{2} \right \}.
\end{equation*}
The bottleneck ratio~\cite[Section~7.2]{LevPerWil:Mixing}
for the partition $\{E,E^\text{\tiny c}\}$ is
\begin{align*}
\Phi(E,E^\text{\tiny c}) &= \frac{\pi \left (\frac{L}{2}\right ) F_{\frac{L}{2}, \frac{L}{2}+1 } + \pi(1) F_{1,0}}{\frac{1}{2}} \\
&\lesssim \frac{\exp(-L) }{1 + \exp \left [L \left ( V \left (\frac1L \right) -   V \left ( 0 \right) \right )\right ]} \\
&\lesssim \exp(-L).
\end{align*}
(As in the last section, we use the symbols $\lesssim$ and $\gtrsim$ to denote bounds up to multiplicative constants.)
Therefore, by the bottleneck inequality,
the mixing time $t_{\mathrm{mix}}$~\cite[Section~4.5]{LevPerWil:Mixing} satisfies
\begin{equation*}
t_{\mathrm{mix}} \geq \frac{1}{4\Phi(E,E^\text{\tiny c})} \gtrsim \exp(L).
\end{equation*}
By~\cite[Theorem~12.3]{LevPerWil:Mixing},
\begin{equation*}
  \frac{1}{\gamma} \log \left (\frac{4}{\min \{ \pi(i):i \in \Omega\}} \right )
= \frac{\log(4) + L}{\gamma} \geq t_{\mathrm{mix}}.
\end{equation*}
Therefore,
\begin{equation}\label{eq: lower bound spectral gap}
  \frac{1}{\gamma} \gtrsim \frac{\exp(L)}{L},
\end{equation}
and we see that the spectral gap decreases exponentially in $L$.

We now relate the condition numbers to the spectral gap.
Using Equation~(3.3) and the table at the bottom of page~147 in~\cite{ChoMey:Survey},
\rev{and also~\cite[Corollary~2.6]{Kirkland:OnConditionNumsMarkov2002}},
we have
\begin{equation}\label{eq: meyer bound on cond numbers}
  \kappa_i \gtrsim \frac{1}{L \min \{|1-\lambda|: \lambda \in \sigma(F) \setminus \{1\}\} }
\end{equation}
for all $i = 1, \dots, 8$.
Now we claim that \rev{for sufficiently large $L$},
\begin{equation}\label{eq: abs spec gap and spec gap are same}
\gamma  = 1 - \max\{\lvert \lambda \rvert : \lambda \in \sigma(F) \setminus \{1\}\}
=  \min \{|1-\lambda|:\lambda \in \sigma(F) \setminus \{1\}\},
\end{equation}
in which case~\eqref{eq: lower bound spectral gap} and~\eqref{eq: meyer bound on cond numbers}
imply that all condition numbers grow exponentially with $L$.
To see this, we first observe that since $F$ is reversible,
its spectrum is real~\cite[Lemma~12.2]{LevPerWil:Mixing}.
Moreover,
\begin{equation*}
  F_{ii} \geq \frac{1}{1 + \exp({\rm Lip}( V))} = \frac{1}{1+\exp(1)} >0  \mbox{ for all } i \in \Omega,
\end{equation*}
where ${\rm Lip}(V)=1$ is the Lipschitz constant of
$V$.
Therefore, using the Gershgorin circle theorem we have
\begin{equation}\label{eq: lower bound on evals}
\lambda \geq \frac{2}{1 + \exp(1)} - 1
\end{equation}
for all $\lambda \in \sigma(F)$.
Inequality~\eqref{eq: lower bound on evals} shows that
$\sigma(F)$ is bounded above $-1$ uniformly in $L$,
and by~\eqref{eq: lower bound spectral gap},
\begin{equation*}
\lim_{L\rightarrow \infty} \gamma
= \lim_{L \rightarrow \infty}  1 - \max\{\lvert \lambda \rvert : \lambda \in \sigma(F) \setminus \{1\}\}= 0.
\end{equation*}
It follows that for sufficiently large $L$,
$\max\{\lvert \lambda \rvert : \lambda \in \sigma(F) \setminus \{1\}\}$ is attained for $\lambda >0$.
Thus, equation~\eqref{eq: abs spec gap and spec gap are same} holds,
and we conclude using~\eqref{eq: lower bound spectral gap}
and~\eqref{eq: meyer bound on cond numbers} that
\begin{equation*}
 \kappa_i \gtrsim \frac{\exp(L)}{L^2}
\end{equation*}
for all $i = 1, \dots, 8$.

\subsection{Bounds below by a random walk}\label{sec: bound below by random walk}
Let $Y$ be the random walk on $\Omega$ with transition matrix
\begin{alignat*}{3}
P_{ii}  = P_{i, i+1} = P_{i, i-1} &:= \frac{1}{3} &&\mbox{ for all } i\in \Omega, \mbox{ and }\\
P_{ij} &:= 0 &&\mbox{ otherwise.}
\end{alignat*}
(As above, since $\Omega$ has periodic boundaries,
$P_{L,L+1}$ means $P_{L,1}$, etc.)
In this section, we use Theorem~\ref{thm: global bound diff sparsity}
to relate $\Q_{ij}(P)$
with $\Q_{ij}(F)$ for $F$ the hilly landscape transition matrix.
First, using the lower bounds on entries of $F$
derived in Section~\ref{sec: calculations and figs},
we have
\begin{equation*}
 F \geq  \frac{3}{2} \frac{1}{1 + \exp(1)} P.
\end{equation*}
Therefore, by~\eqref{eq: lower bound of random walk sens} and Lemma~\ref{lem: monotonicity},
\begin{equation}\label{eq: random walk lb on hilly landscape F}
\Q_{ij} (\alpha F) \geq \Q_{ij} \left( \frac{3\alpha}{2 (1 + \exp(1))} P \right ).
\end{equation}
Now for any $\beta \in (0,1)$,
\begin{align*}
\Q_{ij}(\beta P) &= \P_i[\tau_j < \min \{\tau_i, \tau_\omega\}](\beta P^\omega) \\
&= \sum_{M=1}^\infty \P_i[\tau_j = M, \min \{\tau_i, \tau_\omega\} > M](\beta P^\omega).
\end{align*}
Let $\lvert i - j \rvert$ denote the minimum number of time steps
required for the chain to reach state $j$ from state $i$.
Adopting the notation used in the proof of Lemma~\ref{lem: monotonicity},
there is some path $\gamma \in \mathscr{P}_{ij}^{\lvert i - j \rvert}$
of length $\lvert i -j \rvert$ for which
\begin{align}
\Q_{ij}(\beta P)
&\geq \P_i[\tau_j = \lvert i -j \rvert, \min \{\tau_i, \tau_\omega\} > \lvert i -j \rvert]
(\beta P^\omega) \nonumber \\
&\geq \P_i[X_k = \gamma(k) \text{ for } k = 1, \dots, \lvert i-j \rvert](\beta P^\omega)
\nonumber \\
&=\left ( \frac{\beta}{3} \right )^{\lvert i - j \rvert}.
\label{eq: lower bound of random walk sens}
\end{align}
Combining~\eqref{eq: random walk lb on hilly landscape F}
and~\eqref{eq: lower bound of random walk sens} then yields
\begin{equation*}
\Q_{ij} (\alpha F) \geq \Q_{ij} \left( \frac{3\alpha}{2 (1 + \exp(1))} P \right )
\geq \left ( \frac{3\alpha}{2 (1 + \exp(1))} \right )^{\lvert i-j \rvert}.
\end{equation*}

\section{Acknowledgements}
This work was funded by the NIH under grant 5 R01 GM109455-02.
We would like to thank Aaron Dinner, Jian Ding, Lek-Heng Lim, and Jonathan Mattingly
for many very helpful discussions and suggestions.

\appendix
\section{Proof of Lemma~\ref{lem: inv dist is differentiable}}
\label{app: proof inv dist is differentiable}

\begin{proof}
Let $F$ be irreducible and stochastic.
By~\cite[Equation~(3.1)]{GolMey:ComputingInvDist}, ${\det(I - F_i) > 0}$ for all $i \in \Omega$, and
\begin{equation}\label{eq: principal minor formula for pi}
  \pi(F)^\t = \frac{1}{\sum_{i=1}^L \det(I- F_i)} (\det(I-F_1), \det(I-F_2), \dots, \det(I-F_L)).
\end{equation}
The right hand side of~\eqref{eq: principal minor formula for pi}
yields the desired extension.
To show this, we first observe that there exists an open neighborhood $\mathcal{V}_F \subset \Real^{L \times L}$
of $F$ and a disc $\mathcal{D}_F \subset \mathbb{C}$ with $1 \in \mathcal{D}_F$ such that $G \in \mathcal{V}_F$ implies
(a) $\sum_i \det(I-G_i) >0$ and (b) $G$ has exactly one eigenvalue in $\mathcal{D}_F$ and that eigenvalue is simple.
There exists a neighborhood with property $(a)$,
since $\sum_i \det(I-G_i)$ is continuous in $G$
and $\sum_i \det(I-F_i) >0$.
The existence of a neighborhood $\mathcal{V}_F$ and disc $\mathcal{D}_F$
with property (b) follows from standard results in perturbation theory,
since $F$ irreducible and stochastic implies that 1 is a simple eigenvalue of $F$;
see~\cite[Ch.\ II, Theorem~5.14]{Kato:PertLinearOps}.
We now let $\mathcal{U}$ be the union of the sets $\mathcal{V}_F$ over all irreducible, stochastic $F$,
and we define $\bar{\pi} : \mathcal{U} \rightarrow \Real^L$ by extending the formula
on the right hand side of~\eqref{eq: principal minor formula for pi}.
By (a), $\bar{\pi}$ is continuously differentiable.
By property (b), we know that if $G \in \mathcal{U}$ with $Ge =e$ then $e$ is a simple eigenvalue of $G$,
and so $\ker (I-G) = \Real e$.
Following the proof of~\cite[Theorem~3.1]{GolMey:ComputingInvDist},
one may then use the identity
\begin{equation*} (\adj (I-G))(I-G) = (I-G) (\adj (I-G)) =0,\end{equation*}
where $\adj(\cdot)$ denotes the adjugate matrix, to show that $\bar{\pi}(G)^t G = \bar{\pi}(G)^\t$.
\end{proof}


\end{document}